\documentclass[11pt, reqno]{amsart}
\usepackage[numbers, square]{natbib}
\usepackage[utf8]{inputenc} 
\usepackage{url}            
\usepackage{nicefrac}       
\usepackage{microtype}      
\usepackage[shortlabels]{enumitem}

\usepackage{algorithm2e}

\usepackage[parfill]{parskip}
\usepackage
{
	bbm,
	graphicx,
	amssymb,
	amsmath,
	amsthm,
	dsfont, 
	xcolor,
	enumitem,
	mathtools,
	ifthen,
	mdframed,
	multirow,
	subcaption,
	comment,
    todonotes,
    bm,
}

\usepackage[hyperindex,
            breaklinks,
            colorlinks=true,
            citecolor=blue, 
            urlcolor=blue, 
            linkcolor = blue]{hyperref}

\usepackage{fge}

\usepackage{tikz-cd}
\tikzcdset{every label/.append style = {font = \small}}

\usepackage{mathrsfs}

\makeatletter
\newtheorem*{rep@theorem}{\rep@title}
\newcommand{\newreptheorem}[2]{%
	\newenvironment{rep#1}[1]{%
		\def\rep@title{#2 \ref{##1}}%
		\begin{rep@theorem}}%
		{\end{rep@theorem}}}
\makeatother

\numberwithin{equation}{section}

\newcommand{\N}{\mathbb{N}}                                     
\newcommand{\R}{\mathbb{R}}                                     
\newcommand{\Z}{\mathbb{Z}}                                     

\newcommand{\dd}{\mathrm{d}}                                    

\providecommand{\abs}[1]{\left\lvert #1 \right\rvert}           
\providecommand{\norm}[1]{\left\lVert #1 \right\rVert}          
\newcommand{\innerprod}[2]{\left\langle #1,\, #2 \right\rangle} 
\newcommand{\tr}[1]{\mathrm{Tr} \left( #1 \right)}              

\newcommand\restr[2]{{\left.\kern-\nulldelimiterspace #1 \vphantom{\big|} \right|_{#2}}} 

\DeclareMathOperator{\domain}{dom}
\DeclareMathOperator{\range}{range}

\newcommand{\Ex}{\mathbb{E}}
\renewcommand{\Pr}{\mathbb{P}}

\newcommand{\sigalg}{\mathcal{F}}
\newcommand{\setM}{\mathcal{M}}

\newcommand{\law}{\mathcal{L}}

\newcommand{\invmeas}{\pi} 
\newcommand{\tk}{p} 

\newcommand{\proj}{\Pi}

\newcommand{\reg}{\lambda} 

\newcommand{\bx}{\mathbf{x}} 
\newcommand{\bz}{\mathbf{z}} 

\newcommand{\risk}{{R}} 
\newcommand{\regrisk}{\risk_{\reg}} %
\newcommand{\empregrisk}{\risk_{\reg,\bz}}

\newcommand{\Fstar}{F_p} 
\newcommand{\FG}{F_\vecrkhs} 
\newcommand{\regF}{F_{\reg}} 

\newcommand{\regA}{A_{\reg}} 

\newcommand{\empregF}{F_{\reg, \bz}}

\newcommand{\empregA}{A_{\reg, \bz}}

\newcommand{\inrkhs}{{\mathscr{H}}} 

\DeclareMathOperator{\mmd}{MMD}

\newcommand{\infe}[1][]{
	\ifthenelse{\equal{#1}{}}{{\Phi}}{{\Phi}_{\scriptscriptstyle #1}}}

\newcommand{\outfe}[1][]{
	\ifthenelse{\equal{#1}{}}{{\Psi}}{{\Psi}_{\scriptscriptstyle #1}}}

\newcommand{\outreffe}[1][]{
	\ifthenelse{\equal{#1}{}}{{\Gamma}}{{\Gamma}_{\scriptscriptstyle #1}}}

\newcommand{\vecrkhs}{\mathscr{G}}
\newcommand{\veckernel}{K}

\newcommand{\inrkhsi}{i}
\newcommand{\vecrkhsi}{\mathcal{I}}

\newcommand{\vecrkhscov}{T}

\newcommand{\samplingop}{S_\bx}
\newcommand{\empvecrkhscov}{\vecrkhscov_{\bx}}

\newcommand{\bounded}[1]{\mathfrak{B}(#1)}
\newcommand{\HS}[1]{\mathrm{S_2}(#1)}

\newcommand{\comp}{\Xi}

\newcommand{\id}{\mathrm{Id}}

\newcommand{\idop}[1][]{
	\ifthenelse{\equal{#1}{}}{{\mathcal{I}}}{{\mathcal{I}}_{\scriptscriptstyle #1}}}

\newcommand{\ebdO}[1][]{
	\ifthenelse{\equal{#1}{}}{\mathcal{E}}{\mathcal{E}_{#1}}}
\newcommand{\pro}[1][]{
	\ifthenelse{\equal{#1}{}}{\mathcal{Q}}{\mathcal{Q}_{#1}}}

\newcommand{\cov}[1][]{C_{#1}} 
\newcommand{\empcov}[1][]{\widehat{\cov}_{ #1}} 

\newcommand{\pf}[1][]{
   \ifthenelse{\equal{#1}{}}{\mathcal{P}}{\mathcal{P}_{#1}}}
\newcommand{\epf}[1][]{
   \ifthenelse{\equal{#1}{}}{\widehat{\mathcal{P}}}{\widehat{\mathcal{P}}_{#1}}}
\newcommand{\ko}[1][]{
   \ifthenelse{\equal{#1}{}}{P}{P_{#1}}}
\newcommand{\eko}[1][]{
   \ifthenelse{\equal{#1}{}}{\widehat{\mathcal{K}}}{\widehat{\mathcal{K}}_{#1}}}

\newcommand{\gram}[1][]{
	\ifthenelse{\equal{#1}{}}{{G}}{{G}_{\scriptscriptstyle #1}}}

\newcommand{\outgram}[1][]{
	\ifthenelse{\equal{#1}{}}{{L}}{{L}_{\scriptscriptstyle #1}}}

\newcommand{\ingram}[1][]{
	\ifthenelse{\equal{#1}{}}{{K}}{{K}_{\scriptscriptstyle #1}}}

\DeclareMathOperator{\mspan}{span}

\DeclareMathOperator*{\argmin}{arg\,min} 

\mathtoolsset{centercolon}

\allowdisplaybreaks

\newif\ifcomments

\commentstrue 

\ifcomments 

\newcommand{\mattes}[1]{{\color{red}{\textbf{Mattes}: #1} }}

\else

\newcommand{\mattes}[1]{{}}

\excludecomment{bmattes}
\fi

\theoremstyle{plain}

\newtheorem{theorem}{Theorem}[section]
\newreptheorem{theorem}{Theorem}

\newtheorem{corollary}[theorem]{Corollary}
\newreptheorem{corollary}{Corollary}

\newreptheorem{lemma}{Lemma}

\newtheorem{definition}[theorem]{Definition}

\theoremstyle{definition}

\theoremstyle{definition}
\newtheorem{assump}{Assumption}

\theoremstyle{remark}
\newtheorem{remark}[theorem]{Remark}
\newtheorem{example}[theorem]{Example}

\calclayout

\title[]{Nonparametric approximation\\of conditional expectation operators}

\usepackage{amsaddr}
\author{Mattes Mollenhauer}
\address{\vspace{-.6cm}Freie Universität Berlin}
\email{mattes.mollenhauer@fu-berlin.de}

\author{\vspace{-1cm}Péter Koltai}
\address{\vspace{-.6cm}Universität Bayreuth}
\email{peter.koltai@uni-bayreuth.de}

\begin{document}
  
	\begin{abstract}
		Given the joint distribution of two random variables $X$ and $Y$ 
		on some second countable locally compact Hausdorff space,
		we investigate the statistical approximation of 
		the $L^2$-operator $\ko$ defined by 
		${[ \ko f](x) := \Ex[ f(Y) \mid X = x ]}$
		under minimal assumptions.
		By modifying its domain, we
		prove that $\ko$ can be arbitrarily well approximated in operator norm
		by Hilbert--Schmidt operators acting on a
		reproducing kernel Hilbert space.
		This fact allows to
		estimate $\ko$ uniformly 
		by finite-rank operators over a dense subspace
		even when $\ko$ is not compact. 
		In terms of modes of convergence,
		we thereby obtain the superiority of
		kernel-based techniques over classically used
		parametric projection approaches such as Galerkin methods. 
		This also provides a novel perspective on which
		limiting object the nonparametric estimate of $\ko$ converges to.
		As an application, 
		we show that these results are particularly important
		for a large family of spectral analysis techniques for Markov
		transition operators.
		Our investigation also gives a 
		new asymptotic perspective on the so-called
		kernel conditional mean embedding, which is the theoretical foundation
		of a wide variety of techniques in kernel-based nonparametric inference.	
   \end{abstract}



\maketitle

\section{Introduction}
\label{sec:introduction}

We consider two random variables $X,Y$
taking values in a measurable space
$(E, \sigalg_E)$ where $E$ is a second countable locally compact Hausdorff space
and $\sigalg_E$
its Borel $\sigma$-field.
Let $(\Omega, \sigalg, \Pr)$ be the underlying probability space
with expectation operator $\Ex$.
Let $\invmeas$ denote the 
the pushforward of $\Pr$ under $X$,
i.e., $X \sim \invmeas$ and let
$L^2(E,\sigalg_E,\invmeas;\R) = L^2(\invmeas)$ be the space
of real-valued Lebesgue square integrable functions on $(E, \sigalg_{E})$
with respect to $\invmeas$. Analogously, define $\nu$ as the 
pushforward of $\Pr$ under $Y$ on $E$, i.e, $Y \sim \nu$.
Our goal is to perform a nonparametric estimation
of the \emph{conditional expectation operator}
$\ko \colon L^2(\nu) \rightarrow L^2(\invmeas)$ defined by
\begin{equation*} 
    [\ko f](x) := \Ex[ f(Y) \mid X = x] = \int_E f(y)\, \tk(x, \dd y) ,
\end{equation*}
where $p: E \times \sigalg_{E} \rightarrow \R_+$ is the \emph{Markov kernel}
which describes a \emph{regular version} of
the distribution of $Y$ conditioned on $X$
in terms of
\begin{equation*}
\Pr[ Y \in \mathcal{A} \mid X = x] = 
\int_{\mathcal{A}} \tk(x, \dd y) = \tk(x, \mathcal{A})
\end{equation*}
for all $x \in E$ and events $\mathcal{A} \in \sigalg_E$.
We will introduce additional notation and details
as well as appropriate assumptions in Section~\ref{sec:prelims}.

\subsection*{Scope of this work}
We derive a natural and self-contained theory of the approximation of
$\ko$ over functions in a \emph{reproducing kernel Hilbert space} (RKHS) which
is densely embedded into the domain of $\ko$.
Our analysis shows that the approximation of $\ko$ is strongly connected
to recently developed concepts in RKHS-based inference and statistical learning
such as the 
\emph{kernel mean embedding}~\citep{Berlinet04:RKHS, Smola07Hilbert},
\emph{maximum mean discrepancy}~\citep{Gretton12:KTT,Sejdinovic2013}
and the \emph{conditional mean embedding} 
\citep{Gruen12, Park2020MeasureTheoretic},
which allows to extend our theory to several practical directions
such as hypothesis testing, filtering and spectral analysis for Markov kernels.

We will focus on deriving approximation-theoretic results
instead of a statistical analysis of convergence rates in our
investigation. However, we show that convergence results can be carried over
from recent results for regularized least squares regression
with infinite-dimensional output variables due to
\citet{LiEtAl2022} and \citet{MollenhauerEtAl2022}.

\begin{example}[Data-driven approximation of Markov transition operators]
As a practical application, we argue that our theory provides a statistical model
for a well-known family of numerical spectral analysis techniques for 
\emph{Markov transition operators}, which we highlight in the following example.

The above scenario is of particular practical interest
when $Y := X_{t+\tau}$ and $X := X_t$
for some stationary Markov process $(X_t)_{t \in \R}$ on
the state space $(E, \sigalg_{E})$,
as in this case $\invmeas = \nu$ and $\ko$ given by
\begin{equation}
	\label{eq:markov_transition_operator}
	[\ko f](x) = \Ex[ f(X_{t + \tau}) \mid X_t = x ]
\end{equation}
is the Markov transition operator
with respect to the time lag $\tau >0$.

In the context of Markov processes and dynamical systems, 
it is known that the
spectrum of $\ko$ and the associated eigenfunctions determine
crucial properties of the underlying dynamics such as ergodicity, speed of
mixing, the decomposition of the state space 
into almost invariant (so-called \emph{metastable}) components and many more
~\citep{Davies1982Metastable1,Davies1982Metastable2,Davies1983Spectral, 
Roberts1997Geometric, Roberts2001Geometric, Kontoyiannis2003, Kontoyiannis2005}.


As such, the operator $\ko$ is empirically approximated in 
a wide variety of applied scientific disciplines
by performing a projection onto finite-dimensional subspaces of 
$L^2(\invmeas)$
\citep{%
 Li76,
 DiLi91,
 DJ99,
 JuKo09,
 Schmid10,
 TRLBK14, 
 WKR15, 
 KKS16, 
 KNKWKSN18}.
That is, given an $n$-dimensional subspace 
$V_n \subset L^2(\invmeas) \cap L^2(\nu)$ spanned
by a dictionary of basis elements,
a Monte Carlo quadrature based
on observational data is performed on $V_n$
to obtain the empirical finite-rank operator \smash{$\widehat{\ko}_{n}$}
as an estimate of the \emph{Galerkin-approximation} 
$\ko_n := \Pi_{\invmeas,n} \ko \Pi_{\nu,n}$. Here, 
$\Pi_{\invmeas,n}$ and $\Pi_{\nu,n}$ are the
the orthogonal projection operators onto $V_n$
acting on $L^2(\invmeas)$ and $L^2(\nu)$, respectively.
Under the assumption of ergodicity, one typically obtains
\smash{$\widehat{\ko}_n \to P_n$} almost surely 
by some version of Birkhoff's ergodic theorem~\citep{KKS16}.
From a statistical perpective, these methods can be regarded as
\emph{parametric models}, the parameter choice being
the fixed basis functions spanning the ansatz space $V_n$.
By increasing the number of spanning elements,
a convergence of the Galerkin approximation $P_n$ to
the real operator $\ko$ in only the \emph{strong operator topology}
as $n \to \infty$ can generally be obtained~\citep{KoMe18}.
In practice, the above methods are typically aimed
at performing an empirical spectral analysis of $\ko$,
i.e., spectral properties of \smash{$\widehat{\ko}_n$} are computed
and used as an approximation of the spectral properties of $\ko$.
It is well-known that most desirable spectral convergence results 
require a convergence in operator norm~\citep{Chatelin1983}.
The
spectral convergence of the parametric approaches mentioned above
is therefore ultimately limited by the pointwise convergence of
numerical projection methods~\citep{Hackbusch}.
As a \emph{nonparametric} counterpart of the given parametric methods,
there exist popular RKHS-based versions where
the basis functions are adapted to the 
data~\citep{WRK15:Kernel, Klus2019}.
For these methods, one may hope
that they allow for stronger modes of convergence
than the classical projection methods.
However, it has not been shown yet 
which object is actually approximated
in the infinite-sample limit, as the
asymptotics are significantly more complicated in this case.
Our theory solves this problem and confirms that
the overall convergence is stronger than in the parametric case under
mild assumptions.
The strength of this result comes at the price of 
requiring to restrict the domain of the operator onto an RKHS. 
Whether relevant objects, such as eigenfunctions, of 
the original operator are contained in this space, is in general an open question.
\end{example}

\subsection*{Structure of this paper}
This work is structured as follows.
We delineate related theoretical work
in the field of nonparametric 
statistical inference in Section~\ref{sec:related_work}.
For better accessibility, we present our main results
from a high-level viewpoint in Section~\ref{sec:main_results}.
Section~\ref{sec:prelims} contains the mathematical preliminaries
and detailed assumptions.
We prove our main results in
Section~\ref{sec:nonparametric_approximation} along
some additional findings and elaborate on their
implications from a theoretical perspective.
In Section~\ref{sec:empirical_estimation},
we outline the empirical estimation in the context of
inverse problems and regularization theory, which we
investigate in detail for the 
Tikhonov--Phillips case in Section~\ref{sec:tikhonov_phillips_regularization}.
We revisit our example of Markov transition operators
in Section~\ref{sec:markov_operators} and conclude with
a brief outlook on potential future research in Section~\ref{sec:outlook}.

\section{Related work}
\label{sec:related_work}

This work is inspired by recent development
in RKHS-based statistical inference.
Although our investigation is targeted at
creating a more general mathematical perspective
from an approximation viewpoint,
we make use of the theoretical tools which
were originally developed in this context. 
We therefore highlight the most important work which impacted
our analysis.

Over the last years, the theory of RKHS-based inference
and the 
kernel mean embedding~(KME)\citep{Smola07Hilbert}
spawned a vast variety of methods
in various statistical disciplines.
In this context, a nonparametric
approximation of the conditional mean operation 
$(x,f) \mapsto \Ex[f(Y) \mid X = x]  $ for functions $f$
in some RKHS $\inrkhs$ over $E$ was developed
by \citet{SHSF09} as a purely
linear-algebraic concept under the name 
\emph{conditional mean embedding} (CME). 
This idea has since been used as the theoretical backbone
for methods in Bayesian analysis, graphical models, 
time series analysis, spectral analysis
and dimensionality reduction,
filtering, reinforcement learning and many 
more; we refer to the overview by~\citet{MFSS16} for a 
non-exhaustive selection of applications.

Although the CME as described by \citet{SHSF09} performs well in
applications, the mathematical assumptions imposed in the original work
are typically violated; this has been thoroughly 
examined by~\citet{Klebanov2019rigorous}.
The foundational problems in the theory of the CME led to an investigation of
the approximation of RKHS-valued conditional Bochner expectations
from a regression perspective. In particular, \cite{Gruen12} show that the
empirical Tikhonov--Phillips solution of a regularized least squares regression problem
in a vector-valued reproducing kernel Hilbert space
coincides with the empirical estimate derived by~\citet{SHSF09}. 
Additionally, \cite{Gruen13} propose to use the same estimate for the approximation
of linear operators in a very broad sense
 but do not offer an asymptotic perspective of this idea.

\citet{Park2020MeasureTheoretic} extend the asymptotic 
regression theory of the CME
in the framework of regularised least squares regression in a
\emph{vector-valued reproducing kernel Hilbert space} (vRKHS)
with infinite-dimensional response variable 
\citep{LiEtAl2022, MollenhauerEtAl2022}.
\citet{Klebanov2020linear} extend the operator-theoretic interpretation
of the CME.
In particular, they prove existence of an operator on an RKHS which expresses the
conditional mean under the assumption that the true conditional mean function
is a member of a corresponding tensor product space. 

Concluding the overall picture of the 
aforementioned work: 
while the regression perspective of the 
CME \citep{Gruen12, Park2020MeasureTheoretic} allows to 
consider asymptotic interpretations and prove convergence results, it
has the fundamental drawback that the algebraically interesting
operator-theoretic perspective of $\ko$ is not present. Even more so,
the estimation of spectral properties of $\ko$ (for example in the case
of Markov operators or for dimensionality reduction purposes) is impossible.
Conversely, the operator-theoretic formulation of the CME
\citep{SHSF09, Klebanov2019rigorous, Klebanov2020linear}
lacks an asymptotic perspective and suffers from complex interdependencies of
various assumptions~\citep{Klebanov2019rigorous},
severely impeding a theoretical mathematical analysis. Additionally,
the approximation viewpoint in the $L^2$-operator context has not been investigated yet. We will see that this approximation admits a natural
perspective in terms of the maximum mean discrepancy
between the underlying Markov kernels.

Regarded in the context of the CME,
our results can be interpreted as the missing link between 
the recent work
of~\citet{Klebanov2020linear} and \cite{Park2020MeasureTheoretic}.
In particular, we provide an asymptotic approximation perspective in
the operator-theoretic context of conditional expectations.
On our way, we moreover improve a surrogate risk bound 
used by~\citet{Gruen12} and~\citet{Park2020MeasureTheoretic} 
which serves as the
theoretical foundation for the regression perspective of the CME.
However, our results are formulated in a more general perspective
in terms of the numerical approximation of linear operators 
and can certainly be regarded outside of the context 
of the previously mentioned work on the CME.

\section{Main results}
\label{sec:main_results}

We will briefly outline some main results
and the general content of this work. All discussed 
concepts, mathematical preliminaries and assumptions
will be introduced in Section~\ref{sec:prelims}.
The results presented below are
restated and proven in Section~\ref{sec:nonparametric_approximation}
with more attention to detail---including a careful discussion of
the assumptions.

As previously mentioned, we aim to approximate $\ko$
over a separable reproducing kernel Hilbert space
$\inrkhs$ consisting of functions from $E$ to $\R$ 
generated by the canonical
feature map $\varphi: E \rightarrow \inrkhs$.
We will choose the space  $\inrkhs$ such that is is a subset of 
$C_0(E)$, i.e., the space
of continuous real-valued functions which vanish at infinity~\citep{Carmeli10:universality}.
Additionally, we choose $\inrkhs$ such that it can be continuously embedded
into $L^2(\invmeas)$ as well as $L^2(\nu)$. That is, the inclusion operator 
$\inrkhsi_\invmeas: \inrkhs \rightarrow L^2(\invmeas)$ 
defined by $f \mapsto [f]_{\sim L^2(\invmeas)}$
and the analougously defined inclusion 
$\inrkhsi_\nu : \inrkhs \rightarrow L^2(\nu)$ 
are bounded~\citep{StCh08}. Moreover,
we will generally assume that $\inrkhs$ is 
dense in both $L^2(\invmeas)$ and $L^2(\nu)$.
This property is called $L^2$\emph{-universality}~\citep{Carmeli10:universality,SFL11}.

\begin{remark}[Inclusion operators and notation]
	\label{rem:notation_inclusion}
We will sometimes suppress the inclusion operators
$\inrkhsi_\invmeas$ and $\inrkhsi_\nu$ in our notation
when the context is clear.
In particular, for $f \in \inrkhs$ we will simply write
$\norm{f}_{L^2(\nu)}$ instead of $\norm{\inrkhsi_\nu f }_{L^2(\nu)}$.
Furthermore, under the above assumptions, 
we may understand the operator 
$\ko \inrkhsi_\nu: \inrkhs \rightarrow L^2(\invmeas)$ as a
conditional expectation operator 
acting on functions of $\inrkhs$ via
\begin{equation}
	\label{eq:notation}
	[\ko \inrkhsi_\nu f] (x) = \Ex[f (Y) \mid X = x ] \in L^2(\invmeas) \quad \textnormal{for } f \in \inrkhs
\end{equation}
and use the norm of $\inrkhs$
on its domain. 
By abuse of notation, we may write 
$\ko: \inrkhs \rightarrow L^2(\invmeas)$
instead of $\ko \inrkhsi_\nu$
for the operator in \eqref{eq:notation}.
We will emphasize which version of $\ko$ we refer to
by simply distinguishing between 
$\ko: \inrkhs \rightarrow L^2(\invmeas)$ and
$\ko: L^2(\nu) \rightarrow L^2(\invmeas)$.
We write out the corresponding operator norms
$\norm{\ko}_{ \inrkhs \rightarrow L^2(\invmeas) }$
and
$\norm{\ko}_{L^2(\nu) \rightarrow L^2(\invmeas) }$
to prevent confusion.
Note that by boundedness of $\inrkhsi_\nu$, we have
$\norm{\ko}_{ \inrkhs \rightarrow L^2(\invmeas) } 
\leq 
\norm{\inrkhsi_\nu} \norm{\ko}_{L^2(\nu) \rightarrow L^2(\invmeas) }$.
Similarly, for every bounded operator $A: \inrkhs  \rightarrow \inrkhs$ 
we can consider the bounded operator
$\inrkhsi_\invmeas A$ from $\inrkhs$ to $L^2(\invmeas)$, 
which we will also abbreviate as
 $A: \inrkhs \rightarrow L^2(\invmeas)$.
At this point, it is worth mentioning 
that functions in $\inrkhs$ are generally defined pointwise,
while elements of $L^2(\invmeas)$ are equivalence 
classes of $\invmeas$-a.e.\ equivalent functions.
\end{remark}

\begin{remark}[$P: \inrkhs \to L^2(\invmeas)$ is always Hilbert--Schmidt]
\label{rem:hs}
It is known that under the assumptions above, the 
inclusions $\inrkhsi_\invmeas$ and $\inrkhsi_\nu$
are Hilbert--Schmidt operators~\citep[Chapter 4.3]{StCh08}.
Therefore, the operator $\ko: \inrkhs \rightarrow L^2(\invmeas)$
is Hilbert--Schmidt (and hence compact), independently of the fact whether
$\ko: L^2(\nu) \rightarrow L^2(\invmeas)$ is Hilbert--Schmidt or not.
Intuitively, the approximation of $\ko$ over functions in $\inrkhs$ in operator norm is
therefore generally possible with finite-rank operators from $\inrkhs$ to $L^2(\invmeas)$.
Since we can not efficiently impose the class of Hilbert--Schmidt operators
from $\inrkhs$ to $L^2(\invmeas)$ as a nonparametric hypothesis space in practical applications,
we now provide a more suitable approximation theory for practical scenarios.
The following result  shows that we may actually restrict ourselves to the class
of Hilbert--Schmidt operators mapping from the space $\inrkhs$ to itself
and still expect 
an approximation of $\ko: \inrkhs \rightarrow L^2(\invmeas)$
up to an arbitrary degree of accuracy.
\end{remark}

\begin{theorem}[Approximation by Hilbert--Schmidt operators]
	\label{thm:nonparametric_approximation}
	If there exists a reproducing kernel Hilbert space $\inrkhs \subset C_0(E)$
	which is densely and continuously embedded
	into both $L^2(\invmeas)$ and $L^2(\nu)$,
	then  for every $\delta > 0$, there exists a Hilbert--Schmidt 
    operator $A \colon \inrkhs \rightarrow \inrkhs$, such that
	\begin{equation*}
		\norm{A - \ko}_{\inrkhs \rightarrow L^2(\invmeas)} < \delta.
	\end{equation*}
\end{theorem}

\begin{remark}
	Some remarks related to Theorem~\ref{thm:nonparametric_approximation}
	are in order.
	\begin{enumerate}
	\item We do not require $\ko: L^2(\nu) \rightarrow L^2(\invmeas)$ 
	to be a Hilbert--Schmidt operator
	or compact in order 
	for the above statement to hold.
	Our result is not a contradiction to the known
	fact that operator norm
	limits of Hilbert--Schmidt operators are compact.
	The reason for that is that
	the compactness property is given with respect
	to the norm $\norm{\cdot}_{\inrkhs}$ on the domain,
	which is stronger than the norm $\norm{\cdot}_{L^2(\nu)}$.
	Hence, the continuous extension
	to $A: L^2(\nu) \rightarrow \inrkhs$
	via the known construction
	for bounded operators~\citep[Theorem 4.5]{Weidmann} is generally
	not compact.
	This can equivalently be seen by the fact that $\inrkhsi_\nu$ does
	generally not admit a globally defined bounded inverse.
	We visualize Theorem~\ref{thm:nonparametric_approximation}
	in Figure~\ref{fig:nonparametric_approximation}.
	\item 
	The assumptions on $\inrkhs$
	are not restrictive, as they are
	well examined in statistical learning theory and often satisfied
	for particular RKHSs used 
	in practice. 
	It is actually sufficient to only require that $\inrkhs$ is
	dense in $L^2(\rho)$ for any probability measure
	$\rho$ on $(E, \sigalg_{E})$, as this implies
	denseness in both $L^2(\invmeas)$ and $L^2(\nu)$.
	We address these topics in detail in Section~\ref{sec:prelims}.
	\item We will later also see under which requirements
	there exists a Hilbert--Schmidt operator $A: \inrkhs \rightarrow \inrkhs$
	such that $\norm{A  - \ko}_{\inrkhs \rightarrow L^2(\invmeas)} = 0$.
 	\end{enumerate}
\end{remark}

\begin{figure}[htb]
	\centering

	\begin{tikzcd}[column sep=large, row sep =large]
		L^2(\nu) \arrow[r, "P"]  & L^2(\invmeas)         \\
		\inrkhs \arrow[u, "\inrkhsi_\nu"] 
		\arrow[r, "A \in \HS{\inrkhs}"'] 
		\arrow[ru, "\inrkhsi_\invmeas A"', shift right]  \arrow[ru, "\ko \inrkhsi_\nu", shift left]    & \inrkhs \
		\arrow[u, "\inrkhsi_\invmeas"']
	\end{tikzcd}
	
	\caption{Nonparametric approximation of
		$\ko$ over functions in $\inrkhs$ by a Hilbert--Schmidt operator
		$A \in \HS{\inrkhs}$. Theorem~\ref{thm:nonparametric_approximation}
		shows that 
		$ \ko  \inrkhsi_\nu \approx \inrkhsi_\invmeas A$
		to arbitrary accuracy in the associated operator norm.
		The operator $A$ is approximated by finite-rank operators on $\inrkhs$
		in Corollary~\ref{cor:nonparametric_approximation}.}

	\label{fig:nonparametric_approximation}
\end{figure}

\begin{corollary}
	\label{cor:nonparametric_approximation}
 	Under the assumptions of Theorem~\ref{thm:nonparametric_approximation},
 	there exists a sequence of finite-rank operators $(A_n)_{n \in \N}$
	from $\inrkhs$ to $\inrkhs$ such that 
	$\norm{A_n - \ko}_{\inrkhs \rightarrow L^2(\invmeas)} \to 0$ as $n \to \infty$.
\end{corollary}

As we will prove, 
such a sequence $(A_n)_{n \in \N}$
can be almost surely computed in practice by 
performing a nonparametric regression based on a linear space
consisting of functions mapping from $E$ to $\inrkhs$
given by 
\begin{equation*}
	\vecrkhs = \{  A \varphi(\cdot) : E \rightarrow \inrkhs \, 
	\mid \, A : \inrkhs \rightarrow \inrkhs \textnormal{ is Hilbert--Schmidt} \}.
\end{equation*}
One can show that
the space $\vecrkhs$ is actually a
vector-valued reproducing kernel Hilbert space 
\citep{Carmeli06, Carmeli10:universality} consisting of
$\inrkhs$-valued Bochner square integrable functions.
This fact connects our theory directly
to the aforementioned work on conditional mean embeddings.
We show that the approximation of $\ko$ in the norm
$\norm{\cdot}_{\inrkhs \rightarrow L^2(\invmeas)}$
admits a natural measure-theoretic interpretation
in terms of the well-known 
\emph{maximum mean discprepancy}~\citep{GrettonTwoSample12, Sejdinovic2013}, 
paving the way for nonparametric hypothesis tests
based on $\ko$.

The next result is
the theoretical foundation for Theorem~\ref{thm:nonparametric_approximation}
and will allow us to construct an estimator of $\ko \colon \inrkhs \rightarrow L^2(\invmeas)$
(see Section~\ref{sec:empirical_estimation}).
It shows how the approximation of $\ko$ is related to
the approximation of a conditional Bochner expectation and
improves surrogate risk bounds used by
\citet{Gruen12} and \citet{Park2020MeasureTheoretic}
in the context of the CME
(see Remark~\ref{rem:surrogate_risk} for details).

\begin{theorem}[Regression and conditional mean approximation]
	\label{thm:regression_viewpoint}
	Under the assumptions of Theorem~\ref{thm:nonparametric_approximation},
	we have for every
	Hilbert--Schmidt operator $A \colon \inrkhs \rightarrow \inrkhs$ that
	\begin{equation*}
		\norm{A - \ko}^2_{\inrkhs \rightarrow L^2(\invmeas)} \leq
		\Ex\left[ \norm{ \Fstar(X) - A^*\varphi(X)}_\inrkhs^2   \right] = 
		\norm{ \Fstar - A^*\varphi(\cdot)  }_{ L^2(E, \sigalg_{E}, \invmeas; \inrkhs) }^2,
	\end{equation*}
	where $\Fstar = \Ex[ \varphi(Y) \mid X = \cdot ] \in { L^2(E, \sigalg_{E}, \invmeas; \inrkhs) }$
	is any regular version of
	the $\inrkhs$-valued conditional Bochner expectation 
	$\Ex[ \varphi(Y)  \mid X ] \in L^2(\Omega, \sigalg, \Pr; \inrkhs)$.
	The given bound is sharp.
\end{theorem}

\begin{remark}
	In fact, the above result actually holds under less strict
	assumptions, which we will see in Section~\ref{sec:nonparametric_approximation}.
\end{remark}

As is well-known in statistical learning theory,
the right hand side of the bound in Theorem~\ref{thm:regression_viewpoint} 
is exactly the 
so-called \emph{excess risk} $\risk(F) - \risk(\Fstar)$ of the
infinite-dimensional least squares regression problem
of finding
$\argmin_{F \in \vecrkhs} R(F)$, where
\begin{equation*}
	\risk(F):= \Ex \left[  \norm{ \varphi(Y) - F(X) }^2_\inrkhs \right]
	\textnormal{ for } F(\cdot) = A^*\varphi(\cdot).
\end{equation*}
In particular, the risk $\risk(F)$ allows for the decomposition
\begin{equation*}
	\risk(F) = \norm{ \Fstar - F  }_{ L^2(E, \sigalg_{E}, \invmeas; \inrkhs) }^2 
    + \risk(\Fstar)
\end{equation*}
with the irreducible error term $\risk(\Fstar)$.
This puts the approximation of $\ko$ in line with
the formalism developed for 
regularized least squares regression
with reproducing 
kernels which was established
in a series of highly influential papers
\citep{Caponnetto2007, BauerEtAl2007, YaoEtAl07} 
and its connection to inverse problems in Hilbert spaces.
We refer the reader to
\citet{Muecke18} and the references therein for a more detailed overview.

In particular, by employing a generic \emph{regularization strategy}
$g_\reg$
for a regularization parameter $\reg >0$,
such as for example Tikhonov--Phillips regularization, 
spectral cutoff or Landweber iteration 
(see \citet{EHN96}),
we obtain a regularized solution to the above regression problem
via 
\begin{equation*}
	\label{eq:regularized_solution}
	\regF := 
	g_\reg( \vecrkhscov ) \vecrkhsi^*_\invmeas \Fstar \in \vecrkhs,
\end{equation*}
where $\vecrkhscov: \vecrkhs \rightarrow \vecrkhs$ 
is the \emph{generalized kernel covariance operator} 
(see Section~\ref{sec:vecrkhs_integral_operators})
of the space $\vecrkhs$ 
associated with $X$ and $\vecrkhsi_\invmeas: \vecrkhs \rightarrow 
L^2(E, \sigalg_{E}, \invmeas; \inrkhs)$ is the \emph{inclusion operator} of
$\vecrkhs$ into the space of Bochner square integrable functions 
$L^2(E, \sigalg_{E}, \invmeas; \inrkhs)$.

Since $\vecrkhscov$ plays a crucial role in the underlying inverse problem,
we also show that the action of $\vecrkhscov$ on $\vecrkhs$ admits a dual interpretation 
in terms of
composition operators acting on the class of Hilbert--Schmidt operators on $\inrkhs$. 
For the special case that $g_\reg$ describes Tikhonov--Phillips regularization,
this theory lets us obtain a closed form
expression of the regularized solution
in terms of the \emph{kernel covariance operators}
$\cov[XX]$ and  $\cov[XY]$ on $\inrkhs$. We confirm 
this solution to be the adjoint of the CME first derived by \citet{SHSF09}
given by $\regA = (\cov[XX] + \reg \id_\inrkhs)^{-1} \cov[XY]$ without
the limiting assumptions imposed in the original work. Although this
statement does not
come as a surprise, it has never been proven in any of the aforementioned
papers on the CME. Our results can be interpreted as the population analogue
of a similar statement for the empirical case derived by \citet{Gruen12} 
(see Section~\ref{sec:recovering_operator}).

By performing the
empirical discretization of the above operators and problem~\eqref{eq:regularized_solution}
based on a finite set of observations $\bz = ((X_1, Y_1), \dots, (X_n, Y_n))$
sampled iid from $\law(X,Y)$
in terms of the \emph{sampling operator approach}
\citep{Smale2005, Smale2007}, we obtain
a regularized empirical solution $\empregF(\cdot) = \empregA^*\varphi(\cdot)$.
Theorem~\ref{thm:regression_viewpoint} shows 
that the 
convergence $\empregF \to \Fstar$ 
in $L^2(E, \sigalg_{E}, \invmeas; \inrkhs)$ for $n \to \infty$ 
with a suitable regularization scheme
$\reg = \reg(n)$ implies convergence of $\empregA \to \ko$
in the norm~$\norm{\cdot}_{\inrkhs \rightarrow L^2(\invmeas)}$.

\section{Preliminaries and Assumptions}
\label{sec:prelims}

We give a concise overview of the needed mathematical background.

\subsection{Measure, integration and Hilbert space operators}

We briefly introduce the main concepts from measure theory
and linear operators and analysis in Hilbert spaces.
We refer the reader to~\citet{DiestelUhl1977}, \citet{DunfordSchwartz1, DunfordSchwartz2}
and \citet{Dudley} for details.

For any topological space $E$, we will write $\sigalg_E = \mathcal{B}(E)$ for its associated
Borel field.
For any collection of sets $\setM$,  $\sigma(\setM)$ denotes
the intersection of all $\sigma$-fields containing $\setM$.
For any $\sigma$-field $\sigalg$ and countable index set $I$, 
we write $\sigalg^{\otimes I}$ as the product $\sigma$-field (i.e., the smallest
$\sigma$-field with respect to which all coordinate projections on $E^I$ are measurable).
Note that since
$E$ is Polish (i.e., separable and completely metrizable), we have $\mathcal{B}(E^I) = \mathcal{B}(E)^{\otimes I}$, i.e. the Borel field
on the product space generated by the
product topology and the product of the individual Borel fields are equal.
Put differently, the Borel field operator and the product field operator
are compatible with respect to 
product spaces~\citep[][Proposition 4.1.17]{Dudley}.
Moreover, $E^I$ equipped with the product topology is Polish.

In what follows, we write $B$ for a separable real Banach space
with norm $\norm{\cdot}_B$, and $H$ for a separable real Hilbert space with inner product
$\innerprod{\cdot}{\cdot}_H$. 
The expression $\bounded{B, B'}$ stands for the Banach algebra of bounded 
linear operators from $B$ to another Banach space $B'$ 
and is equipped with the operator norm $\norm{\cdot}$.
For the case $B = B'$,  we abbreviate $\bounded{B, B'} = \bounded{B}$.
We will also write $\norm{\cdot} = \norm{\cdot}_{B \rightarrow B'}$, if the choice
of norms on the underlying spaces $B, B'$ needs
to be emphasized.

Let $(\Omega, \sigalg, \invmeas)$
be a measure space.
For any separable Banach space $B$, we let $L^p(\Omega, \sigalg, \invmeas; B)$ denote the
space of strongly $\sigalg-\sigalg_B$ measurable
and Bochner $p$-integrable functions $f \colon \Omega \rightarrow B$
for $ 1 \leq p \leq \infty$. 
In the case of $B = \R$, we simply write $ L^p(\invmeas):= L^p(\Omega, \sigalg, \invmeas; \R) $ 
for the standard space of real-valued Lebesgue $p$-integrable functions.

The expression $H' \otimes H$
denotes the tensor product of Hilbert spaces $H,H'$.
The Hilbert space $H' \otimes H$ is the completion 
of the algebraic tensor product with respect to the
inner product 
$\innerprod{x_1' \otimes x_1}{x_2' \otimes x_2}_{H' \otimes H} = 
\innerprod{x_1'}{x_2'}_{H'} \innerprod{x_1}{x_2}_H$
for $x_1, x_2 \in H$ and $x_1', x_2' \in H'$.
We interpret the element $x' \otimes x \in H' \otimes H$ 
as the linear
rank-one operator $x' \otimes x \colon H \rightarrow H'$
defined by $\tilde{x} \mapsto \innerprod{\tilde{x}}{x}_H x'$ for all $\tilde{x} \in H$.
Whenever $(e_i)_{i \in I}$, $(e'_j)_{j \in J}$ are complete 
orthonormal systems (CONSs) in $H$ and $H'$,
$(e'_j \otimes e_i)_{i \in I,j \in J}$ is a CONS in $H' \otimes H$.
Thus, when $H$ and $H'$ are separable, $H' \otimes H$ is separable.


For $1 \leq p < \infty$, the \emph{p-Schatten class} $S_p(H, H')$ consists 
of all compact operators $A$ from $H$ to $H'$ such that the
norm $\norm{A}_{S_p(H)} := \norm{(\sigma_i(A))_{i \in J}}_{\ell_p}$ is finite.
Here $\norm{(\sigma_i(A))_{i \in J}}_{\ell_p}$ denotes the $\ell_p$ 
sequence space norm of the sequence of the strictly positive singular values
of $A$ indexed by the countable set $J$, which we assume to be
ordered nonincreasingly.
We set $S_\infty(H, H')$ to be the class of compact operators from $H$ to $H'$ 
equipped with the operator norm and
write $S_p(H) :=S_p(H,H)$ for all $1 \leq p \leq \infty$.
The spaces $S_p(H)$ are two-sided ideals in $\bounded{H}$. 
Moreover $ \norm{A}_{S_q(H,H')} \leq \norm{A}_{S_p(H,H')}$
holds for $1 \leq p \leq q \leq \infty$, i.e., $S_p(H,H') \subseteq S_q(H,H')$.
For $p=2$, we obtain the Hilbert space of \emph{Hilbert--Schmidt operators}
from $H$ to $H'$
equipped with the inner product $\innerprod{A_1}{A_2}_{\HS{H, H'}} = \tr{A_1^*A_2}$.
For $p=1$, we obtain the Banach space
of \emph{trace class operators}. 
The Schatten classes are the completion
of \emph{finite-rank operators} (i.e., operators in 
$\mspan\{ x' \otimes x \mid x \in H, x' \in H' \}$) 
with respect to the corresponding norm.

We will make frequent use of the fact 
that the tensor product space $H' \otimes H$ can be
isometrically identified with the space of Hilbert--Schmidt
operators from $H$ to $H'$, i.e.,
we have $\HS{H, H'} \simeq H' \otimes H$. For elements
$x_1, x_2 \in H$, $x_1', x_2' \in H'$, we have the relation
$\innerprod{x_1' \otimes x_1}{x_2' \otimes x_2}_{H' \otimes H}
= \innerprod{x_1' \otimes x_1}{x_2' \otimes x_2}_{\HS{H,H'}}$,
where the tensors are interpreted as rank-one operators 
as described above. This identification of tensors
with as rank-one operators
extends to $\mspan\{ x' \otimes x \mid x \in H, x' \in H' \}$
by linearity and defines a linear isometric isomorphism
between $H' \otimes H$ and $\HS{H, H'}$, which can also be seen by
considering Hilbert--Schmidt operators in terms of 
their singular value decompositions.
We will frequently switch in between these two viewpoints when considering
Hilbert--Schmidt operators.

\subsection{Joint and regular conditional distributions}

In this paper, we will consider a second countable locally compact Hausdorff
space $(E, \sigalg_{E})$ equipped with its Borel field. We need this technical
setup to avoid dealing with measure-theoretic details later on.

We consider two random variables
$X,Y$ defined on a common probability space 
$(\Omega, \sigalg, \Pr)$ taking values in $E$.
We will assume
without loss of generality that
$(\Omega, \sigalg, \Pr)$ is rich enough to support all performed operations in this paper.
For a finite number of random variables $X_1, \dots, X_n$ defined
with values in $E$,
we write $\law(X_1, \dots, X_n)$ for the \emph{finite-dimensional law},
i.e., \emph{pushforward measure} on $(E^n, \mathcal{B}(E^n))$.
We write \smash{$X \stackrel{d}{=} Y$},
if $X$ and $Y$ are equal in distribution, i.e., their laws are equal.
Throughout this paper, we define
$\invmeas := \law(X)$ and $\nu := \law(Y)$, i.e.,
we have $X \sim \invmeas$ and $Y \sim \nu$.

Let $\tk: E \times \sigalg_{E} \rightarrow \R$ be a
\emph{Markov kernel}\footnote{
	We distinguish different notions of \emph{kernels}
	in this paper.
	We will often refer to reproducing kernels/symmetric positive semidefinite 
	kernels simply as \emph{kernel}, while the kernel $\tk$
	defining a conditional distribution
	will always be called \emph{Markov kernel}.
}, i.e.,
$\tk(x, \cdot)$ is a probability measure on $(E, \sigalg_{E}$) for every $x \in E$
and the map $E \ni x \mapsto \tk(x, \mathcal{A})$ is an $\sigalg_{E}-\R$ measurable function
for every $\mathcal{A} \in \sigalg_{E}$ such that
\begin{equation*}
	\Pr[Y \in \mathcal{A} \mid X = x] = 
	\int_{\mathcal{A}} \tk(x, \dd y) = \tk(x, \mathcal{A})
\end{equation*}
for all $x \in E$ and events $\mathcal{A} \in \sigalg_E$.
The Markov kernel $\tk$ defines a so-called
\emph{regular version}
of the above conditional distribution
which allows to consider the fiberwise disintegration
\begin{equation*}
	\Pr[X \in \mathcal{A}, Y \in \mathcal{B}]
	= \int_{\mathcal{A}}  \tk(x, \mathcal{B}) \, \dd \invmeas(x),
\end{equation*}
see \citet[Theorem 10.2.1]{Dudley}. Such a Markov kernel $p$
exists always in our scenario, since the space $E$ is Polish \citep[Theorem 10.2.2]{Dudley}.
Additionally, two regular versions
of the same conditional distribution
with corresponding Markov kernels $p,p'$
coincide almost everywhere, i.e., we have
$p(x, \cdot) = p'(x, \cdot)$ for $\invmeas$-a.e.\ $x \in E$.

Our goal is to nonparametrically estimate the
\emph{conditional expectation operator}
$\ko \colon L^2(\nu) \rightarrow L^2(\invmeas)$ defined by
\begin{equation*} 
	[\ko f](x) := \Ex[ f(Y) \mid X = x] = \int_E f(y)\, \tk(x, \dd y) ,
\end{equation*} which is a contractive linear map
(and therefore bounded). In fact, this can easily be seen
by making use of Jensen's inequality for conditional expectations
and considering
\begin{align*}
	\norm{\ko f}_{L^2(\invmeas)}^2 = 
	\Ex\left[ \Ex[ f(Y) \mid X ]^2  \right]	
	\leq \Ex\left[ \Ex[ (f(Y))^2 \mid X ] \right]
    = \Ex[ f(Y)^2  ] = \norm{f}_{L^2(\nu)}^2.
\end{align*}

\subsection{Vector-valued reproducing kernel Hilbert spaces}

We will give a brief overview of the concept of a 
\emph{vector-valued reproducing kernel Hilbert space}
(vRKHS), i.e.,
a Hilbert space consisting of
functions from a nonempty set $E$ to a Hilbert space $H$.
Since the construction of such a space is quite technical,
we will not cover all mathematical details here but
rather introduce the most important properties.
For a rigorous treatment of this topic, we refer
the reader to~\citet{Carmeli06} as well as 
\citet{Carmeli10:universality}.

\begin{definition}[Operator-valued psd kernel]
	\label{def:veckernel}
	Let $E$ be a nonempty set and $H$ be a real Hilbert space.
	A function $\veckernel: E \times E \rightarrow \bounded{H}$
	is called an \emph{operator-valued positive-semidefinite (psd)
		kernel}, if 
	$\veckernel(x,x') = \veckernel(x',x)^{*}$ and
	all $x,x' \in E$ and
	additionally for all $n \in \N$,
	$x_1, \dots, x_n \in E$ and $\alpha_1, \dots, \alpha_n \in \R$,
	we have
	\begin{equation*}
		\sum_{i,j=1}^n \alpha_i \alpha_j
		\innerprod{h}{\veckernel(x_i,x_j)h}_H \geq 0
        \quad
        \text{for all } h \in H.
	\end{equation*}
\end{definition}

Let $\veckernel: E \times E \rightarrow \bounded{H}$
be an operator-valued psd kernel.
For a fixed $x \in E$ and $h \in H$, 
we obtain a function from $E$ to $H$ via
\begin{equation*}
	[\veckernel_x h](\cdot) := \veckernel(\cdot,x)h.
\end{equation*}
We can now consider the set
\begin{equation*}
	\vecrkhs_{\mathrm{pre}} :=
	\mspan \{ \veckernel_x h \mid 
	x \in E, h \in H \}
\end{equation*}
and define an inner product on $\vecrkhs_{\mathrm{pre}}$
by linearly extending the expression
\begin{equation}
	\label{eq:vec_innerprod}
	\innerprod{ \veckernel_x h}{\veckernel_{x'} h'}_\vecrkhs
	:=
	\innerprod{ h }{ \veckernel(x,x')h'}_H.
\end{equation}
Let $\vecrkhs$ be the completion of $\vecrkhs_{\mathrm{pre}}$
with respect to this inner product.
We call $\vecrkhs$ the \emph{$H$-valued reproducing kernel Hilbert space}
or more generally the vRKHS induced by the kernel $\veckernel$.

The space $\vecrkhs$ is a Hilbert space consisting of
functions from $E$ to $H$ with the so-called
\emph{reproducing property} given by the identity
\begin{equation*}
	\innerprod{ F(x) }{ h }_H = 
	\innerprod{ F }{ \veckernel_x h}_\vecrkhs
\end{equation*} 
for all $F \in \vecrkhs$, $h \in H$ and $x \in E$.
Additionally, we have 
\begin{equation*}
	\norm{ F(x) }_H \leq 
	\norm{\veckernel(x,x)}^{1/2}
	\, \norm{ F }_\vecrkhs, \quad x \in E
\end{equation*}
for all $F \in \vecrkhs$.
When $\veckernel_x$ is understood as a linear operator
from $H$ to $\vecrkhs$ fixed $x \in E$, the
inner product given by~\eqref{eq:vec_innerprod} implies
that $\veckernel_x$ is a bounded operator for all $x \in E$.
As a result,
we can rewrite the reproducing property as
\begin{equation*}
	F(x) = \veckernel_x^* F 
\end{equation*}
for all $F \in \vecrkhs$ and $x \in E$.
Therefore we obviously have
\begin{equation*}
	\veckernel_x^* \veckernel_{x'} = \veckernel(x,x'),  \quad x,x' \in E
\end{equation*}
and the linear operators $\veckernel_x \colon \inrkhs \rightarrow \vecrkhs$
and $\veckernel_x^* \colon \vecrkhs \rightarrow \inrkhs$
are bounded with 
\begin{equation*}
	\norm{\veckernel_x} = \norm{\veckernel_x^*} = \norm{ K(x,x) }^{1/2}.
\end{equation*}

In this paper, we will deal with two very specific examples
of psd kernels, which we will introduce in what follows.

\subsubsection{$\R$-valued RKHS}
\label{sec:scalar_rkhs}
When we identify the space of linear operators on $\R$
with $\R$ itself and consider a scalar-valued psd kernel
\begin{equation*}
	k \colon E \times E \rightarrow \R
\end{equation*}
in the sense of Definition~\ref{def:veckernel},
we obtain the standard setting 
of the ($\R$-valued) reproducing 
kernel Hilbert space~(RKHS; see \citet{aronszajn50reproducing}).
The kernel $k$ satisfies $k(x,x') = k(x',x)$
for all $x,x' \in E$.
We obtain a space $\inrkhs$ consisting of functions 
from $E$ to $\R$
with the properties
\begin{enumerate}[label=(\roman*), itemsep=0ex, topsep=1ex]
	\item $ \innerprod{f}{k(x, \cdot)}_\mathscr{H} = f(x) $ for all $ f \in \mathscr{H} $ (reproducing property), and
	\item $ \mathscr{H} = \overline{\mspan\{k(x, \cdot) \mid x \in E \}}$, where
	the completion is with respect to the RKHS norm.
\end{enumerate}
It follows in particular that
$ k(x, x^\prime) = \innerprod{k(x, \cdot)}{k(x^\prime,\cdot)}_\mathscr{H} $. 
The so-called \emph{canonical feature map} $ \varphi \colon E \to \mathscr{H} $ 
is given by $ \varphi(x) := k(x, \cdot) $. 

The space $\inrkhs$ has been thoroughly examined over the last decades and 
has numerous applications in statistics, approximation theory and
machine learning. For details, we refer the reader to
\citet{Berlinet04:RKHS} and \citet{StCh08}.

\begin{remark}[Notation]
	In what follows, $\inrkhs$ will always denote
	the $\R$-valued RKHS induced by the kernel
	$k \colon E \times E \rightarrow \R$
	with corresponding canonical feature map 
	$\varphi \colon E \rightarrow \inrkhs$
	as described in this section.
	We will write small letters $f,g,h \in \inrkhs$ for
	$\R$-valued RKHS functions.
\end{remark}

\subsubsection{$\inrkhs$-valued vRKHS}
\label{sec:vecrkhs}
Let $\inrkhs$ be the $\R$-valued RKHS induced by the 
kernel $k: E \times E \rightarrow \R$ as decribed
in Section~\ref{sec:scalar_rkhs}.
Let $\id_\inrkhs$ be the identity operator on $\inrkhs$.
We define the map
$\veckernel \colon E \times E \rightarrow \bounded{\inrkhs}$
with
\begin{equation}
	\veckernel(x, x') := k(x, x') \id_\inrkhs
\end{equation}
for all $x, x' \in E$.
It is straightforward to show that $\veckernel$ is a
psd kernel and therefore induces an $\inrkhs$-valued
vRKHS $\vecrkhs$~\citep[see also][Example 3.3.(i)]{Carmeli10:universality}.

\begin{remark}[Notation]
	In what follows, $\vecrkhs$ will always denote
	the $\inrkhs$-valued vRKHS induced by the kernel 
	$\veckernel \colon E \times E \rightarrow \bounded{\inrkhs}$
	given by $\veckernel(x,x') = k(x,x') \id_\inrkhs$
	as described in this section.
	We will write capital letters $F,G,H \in \vecrkhs$ for
	$\inrkhs$-valued functions
	in order to distinguish them from real-valued functions
	$f,g,h \in \inrkhs$.
\end{remark}

\subsection{Isomorphism between $\vecrkhs$ and $\HS{\inrkhs}$}

The foundation of our approach is given
by the fact that elements of the
vRKHS $\vecrkhs$ 
defined by the kernel $K(x, x') = k(x,x') \id_\inrkhs$ can be interpreted
as Hilbert--Schmidt operators on $\inrkhs$.
We again recall that the space of Hilbert--Schmidt operators $\HS{\inrkhs}$
is isometrically isomorphic to the tensor product space $\inrkhs \otimes \inrkhs$
via an identification of rank-one operators as elementary tensors.
We will use the latter to state the result, since a formulation in this
way is more natural.

\begin{theorem}%
	[$\vecrkhs$ is isomorphic to $\inrkhs \otimes \inrkhs$]
	\label{thm:vecrkhs_isomorphism} 
	Let $\inrkhs$ be a scalar RKHS with corresponding kernel $k$.
	Let $\vecrkhs$ be the vector-valued RKHS induced by the kernel
	$K(x, x') := k(x,x') \id_\inrkhs$.
	The map $\Theta$ defined on rank-one tensors in $\inrkhs \otimes \inrkhs$
	defining an $\inrkhs$-valued function on $E$ by the relation 
	\begin{equation} 
		\label{eq:vecrkhs_isomorphism}
		\left[\Theta(f \otimes h)\right](x) 
		:= h(x) f = (f \otimes h) \varphi(x) = \innerprod{h}{\varphi(x)}_\inrkhs f
	\end{equation}
	for all $x \in E$ and $f,h \in \inrkhs$ maps to $\vecrkhs$.
	Furthermore,
	extending $\Theta$ to $\inrkhs \otimes \inrkhs$ via linearity and completion
	yields an isometric isomorphism
	between $\inrkhs \otimes \inrkhs$ and $\vecrkhs$.
\end{theorem}

A proof of Theorem~\ref{thm:vecrkhs_isomorphism}
can be found in 
\citet[Proposition 3.5 \& Example 3.3(i)]{Carmeli10:universality}.  
The isometric isomorphism
\begin{equation*}
	\Theta: \inrkhs \otimes \inrkhs \rightarrow \vecrkhs
\end{equation*}
defined by~\eqref{eq:vecrkhs_isomorphism}
seems technical but actually becomes quite intuitive
when one examines how
the inner products of both spaces are connected
via the kernels $k$ and $K$.
We outline this connection briefly below.

Let $x,x' \in E$ and $h, h' \in \inrkhs$.
We define $F := K_x h \in \vecrkhs$ and $F':= K_{x'} h' \in \vecrkhs$
and note that we can express the inner product in $\vecrkhs$ as
\begin{align*}
	\innerprod{F}{F'}_\vecrkhs 
	&= \innerprod{ K_{x'}^* K_x h}{h'}_\inrkhs
	= \innerprod{k(x',x) \id_\inrkhs h}{h}_\inrkhs \\
	&= \innerprod{ \varphi(x')}{ \varphi(x) }_{\inrkhs} \, 
	\innerprod{h}{h'}_\inrkhs \\
	&= \innerprod{ h \otimes \varphi(x) }{h' \otimes \varphi(x') }_{\inrkhs \otimes \inrkhs}.
\end{align*}
This derivation
can be extended straightforwardly to a correspondence 
of vector-valued functions 
$F, F' \in \mspan \{ K_x h \, | \, x \in E, h \in \inrkhs \} \subseteq \vecrkhs$
and linear combinations of tensors in
$\{ h \otimes \varphi(x) \, | \, x \in E, h \in \inrkhs \} \subseteq \inrkhs \otimes \inrkhs$
by using bilinearity of the respective inner products.
Since both spans are dense in the associated spaces, this property can
be extended to the full spaces via completion.
We now restate Theorem~\ref{thm:vecrkhs_isomorphism} in a more
accessible way for our scenario. The formulation below shows that
pointwise evaluation of functions  in $\vecrkhs$ may be conducted
by the action of the corresponding operator in $\HS{\inrkhs}$
on the canonical feature map $\varphi$.
We will refer to this property as the \emph{operator reproducing property}.
We visualize the relations between $\inrkhs \otimes \inrkhs$,
$\HS{\inrkhs}$ and $\vecrkhs$ in Figure~\ref{fig:isomorphisms}.

\begin{corollary}[Operator reproducing property]
	\label{cor:vecrkhs_isomorphism} 
	For every function $F \in \vecrkhs$
	there exists an operator $A := \Theta^{-1}(F) \in \HS{\inrkhs}$
	such that
	\begin{equation}
		\label{eq:operator_evaluation}
		F(x) = A \varphi(x) \in \inrkhs
	\end{equation}
	for all $x \in E$ with $\norm{A}_{\HS{\inrkhs}} = \norm{F}_\vecrkhs$
	and vice versa.
	
	Conversely, for any pair $F \in \vecrkhs$ and $A \in \HS{\inrkhs}$ satisfying
	property~\eqref{eq:operator_evaluation} we have   
	$A = \Theta^{-1}(F)$.
\end{corollary}

\begin{proof}
	The first assertion directly follows from Theorem~\ref{thm:vecrkhs_isomorphism} and
	the construction of $\Theta$. It remains to prove the second assertion.
	Let $F \in \vecrkhs$ and define $A:= \Theta^{-1} (F)$.
	By the first assertion, $A$ satifies \eqref{eq:operator_evaluation}.
	Assume there exists 
	$B \in \HS{\inrkhs}$ satisfying \eqref{eq:operator_evaluation}.
	Then by linearity, $A$ and $B$ coincide on $\mspan\{ \varphi(x) \, \mid \, x \in E \}$,
	which is dense in $\inrkhs$. By continuity, we therefore have $A = B$. The operator
	in $\HS{\inrkhs}$ satisfying \eqref{eq:operator_evaluation} is therefore uniquely given
	by $\Theta^{-1}(F)$.
\end{proof}

\begin{figure}[htb]
	\centering
	
	\begin{tikzcd}[column sep=12em]
		\HS{\inrkhs} \arrow[r, leftrightarrow ,"A = \sum_i \sigma_i(A)\, u_i \otimes v_i ", "\textrm{SVD}"'] 
		\arrow[rr,leftrightarrow, "\textrm{Corollary \ref{cor:vecrkhs_isomorphism}}"', bend right, "A \leftrightarrow A\varphi(\cdot)"] 
		& \inrkhs \otimes \inrkhs \arrow[r, leftrightarrow,"\textrm{Theorem~\ref{thm:vecrkhs_isomorphism}}"',
		"\Theta (u_i \otimes v_i) = v_i(\cdot)\, u_i"] 
		& \vecrkhs
	\end{tikzcd}

	
	\caption{Visualization of the isometric isomorphisms
	between $\HS{\inrkhs}$, $\inrkhs \otimes \inrkhs$ and $\vecrkhs$.
	Here, \emph{SVD} refers to the singular value decomposition
	of compact operators.}
	\label{fig:isomorphisms}
\end{figure} 

\begin{remark}[Operator reproducing property]
	Not only does Corollary~\ref{cor:vecrkhs_isomorphism} describe how functions in $\vecrkhs$ can be evaluated in terms
	of their operator analogue in $\HS{\inrkhs}$, it also implies the \emph{implicit} construction
	of $\vecrkhs$ via Hilbert--Schmidt operators acting on the RKHS $\inrkhs$.
	In particular,
	the above result shows that the space of Hilbert--Schmidt operators $\HS{\inrkhs}$
	generates the vRKHS $\vecrkhs$ via
	\begin{equation*}
		\vecrkhs = \{ F: E \rightarrow \inrkhs \mid F = A\varphi(\cdot),\, A \in \HS{\inrkhs} \}.
	\end{equation*}
	Our previous considerations show that $\vecrkhs$ is precisely the vRKHS associated with
	the vector-valued kernel $K := k \id_\inrkhs$.
\end{remark}

Corollary~\ref{cor:vecrkhs_isomorphism}
will be of central importance for our approach.
The identification of an $\inrkhs$-valued vRKHS
function in $\vecrkhs$ with a corresponding Hilbert--Schmidt operator acting on $\inrkhs$
will be used to bridge the gap between vector-valued statistical 
learning theory and the nonparametric estimation of 
linear operators~\citep{Gruen13}.

\subsection{Assumptions on $\inrkhs$}
\label{sec:assumptions}

We impose some technical requirements
on the RKHS $\inrkhs$ and the corresponding kernel $k$.
Our first three assumptions allow that
we can perform Bochner integration 
without being caught up in measurability and integrability issues
later on~\citep{DiestelUhl1977}.
The fourth and the fifth assumption
are needed to ensure that $\inrkhs$ supplies
the typically used approximation qualities in a function
space context.

\begin{assump}[Separability] \label{ass:separability}
	The RKHS $\mathscr{H}$ is separable. Note that for a Polish space $E$, 
	the RKHSs induced by a continuous kernel $k \colon E \times E \rightarrow \R$ is 
	always separable~\citep[Lemma 4.33]{StCh08}.
	For a more general treatment of conditions implying separability, 
    see~\citet{OwhadiScovel2017}.
\end{assump}

\begin{assump}[Measurability] \label{ass:measurability}
	The canonical feature map $\varphi \colon E \rightarrow \inrkhs$ is $\sigalg_E-\sigalg_\inrkhs$ measurable.
	This is the case when $k(x,\cdot) \colon E \rightarrow \R$
	is $\sigalg_E-\sigalg_\R$ measurable for all $x \in E$.
	If this condition holds, then additionally all functions $f \in \inrkhs$ are $\sigalg_E-\sigalg_\R$ measurable 
	and $k \colon E \times E \rightarrow \R$ is $\sigalg_E^{\otimes 2}-\sigalg_\R$ measurable~\citep[Lemmas 4.24 and 4.25]{StCh08}.
\end{assump}

\begin{assump}[Existence of second moments] \label{ass:second_moment}
	We have $\varphi \in L^2(E, \sigalg_E, \invmeas; \inrkhs)$
	as well as $\varphi \in L^2(E, \sigalg_E, \nu; \inrkhs)$.	
	Note that this is equivalent to $\Ex[ \norm{\varphi(X)}_\inrkhs^2  ] < \infty$
	and $\Ex[ \norm{\varphi(Y)}_\inrkhs^2  ] < \infty$
	which trivially holds
	for all probability measures $\invmeas, \nu$ on $(E, \sigalg_{E})$ 
	case whenever $\sup_{x \in E} k(x,x) < \infty$.
\end{assump}

\begin{assump}[$C_0$-kernel]
	\label{ass:c0}
	We assume that $\inrkhs \subseteq C_0(E)$, where
	$C_0(E)$ is the space of continuous real-valued
	 functions on $E$ vanishing at infinity.
	 In particular, this is the case if
	 $x \mapsto k(x,x)$ is bounded on $E$ and
	 $k(x,\cdot) \in C_0(E)$ for all $x \in E$ 
	 \citep[Proposition 2.2]{Carmeli10:universality}.
\end{assump}

\begin{assump}[$L^2$-universal kernel, see
	 Section~\ref{sec:integral_operators}]
	\label{ass:l2}
	We assume that $\inrkhs$ is dense in $L^2(\invmeas)$. In this case, the kernel $k$
	and the RKHS $\inrkhs$ are called $L^2$-\emph{universal}
	 \citep{Carmeli10:universality, SFL11}.
\end{assump}

\begin{remark}
	\label{rem:assumptions}
	Since not all of our results will need all of the above assumptions, 
	we
	collect some general implications of the different assumptions here.
	\begin{enumerate}
	\item Assumptions~\ref{ass:separability}--\ref{ass:second_moment} 
	ensures that
	$\inrkhs$ can be continuously embedded into both $L^2(\invmeas)$ and $L^2(\nu)$ 
	(see Section~\ref{sec:integral_operators}).
	\item The combination
	of Assumption~\ref{ass:c0} and
	Assumption~\ref{ass:l2} implies that $\inrkhs$ is even dense in $L^2(\nu)$
	for all probability measures $\nu$ on $(E, \sigalg_{E})$
	\citep[Theorem 4.1 and Corollary 4.2]{Carmeli10:universality}.

	\item Instead of Assumption~\ref{ass:l2}, it is sometimes
	required in the literature 
	that $\inrkhs$ is dense in $C_0(E)$ with respect 
	to the supremum norm.
	This property is usually called $C_0$\emph{-universality}.
	One can show that when
	Assumption~\ref{ass:c0} holds,
	$C_0$-universality is equivalent to
	$L^2$-universality~\citep{SFL11}.
	\item When Assumptions~\ref{ass:separability}--\ref{ass:l2} 
	are satisfied,
	then the vRKHS $\vecrkhs$ induced by the kernel
	$K = k \id_{\inrkhs}$
	 is dense in both 
	$L^2(E, \sigalg_{E}, \invmeas; \inrkhs)$
	and $L^2(E, \sigalg_{E}, \nu; \inrkhs)$
    \citep[Example 6.3 \& Theorem 4.1]{Carmeli10:universality}.
	This is important for us, as we  will make use of this fact later on.

	\end{enumerate}
\end{remark}

\begin{example}
	For $E \subseteq \R^d$, well-known translation invariant kernels
	such as the \emph{Gaussian kernel} or \emph{Laplacian kernel}
	satisfy 
	all of the above assumptions for arbitrary
	probability measures $\invmeas, \nu$ on $(E, \sigalg_{E})$~\citep{SFL11}.
\end{example}

\subsection{Integral operators and $L^2$-inclusions}
\label{sec:integral_operators}

The Assumptions~\ref{ass:separability}--\ref{ass:second_moment}
imply that $\inrkhs$
can be embedded into spaces of
square integrable functions. 
This fact and its connections to
integral operators defined by the corresponding kernels
plays a fundamental role in learning theory.

\subsubsection{Real-valued RKHS}
We begin with general statements for the scalar kernel 
$k$~\citep[Chapter 4.3]{ StCh08}.
Let the Assumptions~\ref{ass:separability}--\ref{ass:second_moment} be satisfied.
The \emph{inclusion operator} $\inrkhsi_\invmeas: 
\inrkhs \rightarrow L^2(\invmeas)$
given by $ f \mapsto [f]_\sim \in L^2(\invmeas)$
identifies $f \in \inrkhs$ with its equivalence class of $\invmeas$-a.e.\ defined
functions in $L^2(\invmeas)$.
It is bounded with 
$\norm{\inrkhsi_\invmeas} \leq \norm{ \varphi }_{L^2(E,\sigalg_E,\invmeas;\inrkhs)}$
and Hilbert--Schmidt.
The adjoint of $\inrkhsi_\invmeas$ is
the integral operator 
$\inrkhsi^*_\invmeas \colon L^2(\invmeas) \rightarrow \inrkhs$
given by
\begin{equation*}
	[\inrkhsi_\invmeas^* f] (x) = \int_E k(x,x') f (x') \, \dd \invmeas(x'),
	\quad f \in L^2(\invmeas).
\end{equation*}
The kernel $k$ is $L^2$-universal if and only if $\inrkhsi_\invmeas^*$ is injective.

The operator 
$\cov[XX] := \inrkhsi_\invmeas^* \inrkhsi_\invmeas \colon \inrkhs \rightarrow \inrkhs$ is the
\emph{kernel covariance operator}
associated with the measure $\invmeas$
given by
\begin{equation*}
	\cov[XX] = \int_E \varphi(x) \otimes \varphi(x) \, \dd \invmeas(x) 
	= \Ex[ \varphi(X) \otimes \varphi(X) ],
\end{equation*} where the integral converges in trace norm.
We define all of the above concepts analogously for
the measure $\nu$ and the corresponding random variable~$Y$.
The \emph{kernel cross-covariance operator}~\citep{Baker1973}
of $X$ and $Y$
is the trace class operator given~by
\begin{equation*}
	\cov[YX] := \iint_{E \times E} \varphi(y) \otimes \varphi(x) \, p(x, \dd y) \dd \invmeas(x)
	= \Ex[ \varphi(Y) \otimes \varphi(X)] .
\end{equation*}
Both operators satisfy 
$\innerprod{h}{\cov[XX] f}_\inrkhs = \innerprod{h}{f}_{L^2(\invmeas)} = \Ex[f(X) h(X)]$
as well as $ \innerprod{h}{\cov[YX] f}_\inrkhs = \Ex[f(X) h(Y)]$
for all $f,h \in \inrkhs$.

\begin{remark}[Scalar RKHSs and integral operators]
	Although the operators 
    $\inrkhsi_\invmeas^*: L^2(\invmeas) \rightarrow \inrkhs$, 
	$\inrkhsi_\invmeas \inrkhsi_\invmeas^*: L^2(\invmeas) \rightarrow L^2(\invmeas)$ and
	$\cov[XX]: \inrkhs \rightarrow \inrkhs$ have the same 
	analytical expression as integral
	operators, they are fundamentally different objects since
	they operate on different spaces.
	However, $\inrkhsi_\invmeas \inrkhsi_\invmeas^*$ and $\cov[XX]$ 
	share the same nonzero 
	eigenvalues and their eigenfunctions can be 
	related~\citep{RBD10}.
\end{remark}

\subsubsection{Vector-valued RKHS} 
\label{sec:vecrkhs_integral_operators}
Similarly to the above operators defined for the scalar kernel $k$,
we can define the above concepts for the vector-valued kernel 
$K = k \id_{\inrkhs}$ in the context
of Bochner spaces \citep{Carmeli06, Carmeli10:universality}. 

When Assumptions~\ref{ass:separability}--\ref{ass:second_moment}
are satisfied,
the space $\vecrkhs$ is separable.
The elements of $\vecrkhs$ are
$\sigalg_E -\sigalg_\inrkhs$ measurable functions.
Additionally, they are Bochner square integrable w.r.t. $\invmeas$.
The inclusion operator 
$\vecrkhsi_\invmeas \colon \vecrkhs \rightarrow L^2(E,\sigalg_E,\invmeas;\inrkhs)$
given by $F \mapsto [F]_\sim $ is bounded with
$\norm{\vecrkhsi_\invmeas} \leq \norm{ \varphi }_{L^2(E,\sigalg_E,\invmeas;\inrkhs)}$.

The adjoint of $\vecrkhsi_\invmeas$ is
the integral operator 
$\vecrkhsi_\invmeas^* \colon L^2(E, \sigalg_E, \invmeas, \inrkhs) \rightarrow \vecrkhs$
given by
\begin{equation*}
	[\vecrkhsi_\invmeas^* F] (x) = \int_E \veckernel(x,x') F (x') \, \dd \invmeas(x'),
	\quad F \in L^2(E, \sigalg_E, \invmeas, \inrkhs).
\end{equation*}
The operator 
$\vecrkhscov := \vecrkhsi_\invmeas^* \vecrkhsi_\invmeas \colon \vecrkhs \rightarrow \vecrkhs$
is the \emph{generalized covariance operator}~(also called 
\emph{frame operator}, \citet{Carmeli06})
associated with the measure $\invmeas$
given by
\begin{equation*}
	\vecrkhscov F
	= \int_E K_x K_x^* F \, \dd \invmeas(x)
\end{equation*}
for all $F \in \vecrkhs$. $\vecrkhscov$ is bounded.

The following example shows that the generalized covariance operator 
$\vecrkhscov$
associated with $K(x,x') = k(x,x') \id_\inrkhs$ 
is noncompact in general. 
For more details, we refer the reader to recent recent
results by \citet{MollenhauerEtAl2022} characterising
the spectum of $\vecrkhscov$.

\begin{example}[Noncompact generalized covariance operator $\vecrkhscov$]
	\label{ex:noncompact} It is easy to
	see that for commonly used radial kernels $k$ such as the Gaussian
	kernel on $E \subseteq \R^d$, the generalized covariance operator
	$\vecrkhscov$ is never compact.\\
	Consider a measurable kernel
	$k: E \times E \rightarrow \R$ which induces an infinite-dimensional RKHS
	$\inrkhs$
	satisfying Assumptions~\ref{ass:separability} and \ref{ass:measurability}. 
	Assume $k(x,y) > 0$ for all $x,y \in E$ and  $k(x,x) = 1$ for all $x \in E$.
	Let $K = k \id_\inrkhs$ and $(e_i)_{i \in \N} \subset \inrkhs$ be an ONS. 
	We fix some $x' \in E$ and
	define $F_i := K_{x'} e_i \in \vecrkhs$ for all $i \in \N$. Note that we have
	\begin{equation*}
		\innerprod{K_{x'} e_i }{K_{x'} e_j }_\vecrkhs = 
		\innerprod{k(x',\cdot) e_i}{k(x', \cdot) e_j}_\vecrkhs 
        = k(x', x') \innerprod{e_i}{e_j}_\inrkhs = \delta_{ij},
	\end{equation*}   
	i.e., $(F_i)_{i \in \N} $ is an ONS in $\vecrkhs$.
	Then it is possible to show that $(\vecrkhscov F_i)_{ i \in \N}$ 
	consists of orthogonal elements of the same length:
	\begin{align*}
		\innerprod{ \vecrkhscov F_i}{\vecrkhscov F_j}_\vecrkhs
		&=
		\innerprod{\int_E K_x F_i(x) \dd \invmeas(x) }{ \int_E K_x F_j(x) \dd \invmeas(x) }_\vecrkhs \\
		&=
		\innerprod{\int_E k(x', x) \, K_x e_i \dd \invmeas(x) }{ \int_E k(x', x) \, K_x e_j  \dd \invmeas(x) }_\vecrkhs \\
		&= 
		\iint_{ E^2 } k(x',x) k(x',y) \innerprod{  K^*_y K_x e_i }{ e_j}_\inrkhs	\dd [\invmeas \otimes \invmeas] (x,y) \\
		&= 	\iint_{ E^2 } k(x',x) k(x',y) k(x,y) \innerprod{ e_i }{ e_j}_\inrkhs	\dd [\invmeas \otimes \invmeas] (x,y)  = M \delta_{ij}
	\end{align*}
	with the constant $M := \iint_{ E^2 } k(x',x) k(x',y) k(x,y) \,	\dd [\invmeas \otimes \invmeas] (x,y) >0$, which is
	independent of $i,j \in \N$.
	Consequently, we
	have
	$ \norm{\vecrkhscov F_i - \vecrkhscov F_j}^2_\vecrkhs  =  
	\norm{\vecrkhscov F_i }^2_\vecrkhs + \norm{\vecrkhscov F_j }^2_\vecrkhs = 2M$
	for all $i \neq j$,
	i.e., no subsequence of $(\vecrkhscov F_i)_{ i \in \N}$ can be Cauchy.
	We therefore have constructed a bounded sequence $ (F_i)_{i \in \N} $ in $\vecrkhs$
	such that $ (\vecrkhscov F_i)_{i \in \N} $ does not contain a convergent subsequence in $\vecrkhs$, implying that $\vecrkhscov$
	is not compact.
\end{example}


\subsection{Conditional mean embeddings and regression function}

Under Assumptions~\ref{ass:separability}--\ref{ass:second_moment},
the Bochner integrability of the feature map $\varphi: E \rightarrow \inrkhs$ 
can be elegantly used
in combination with the reproducing property of $\inrkhs$
to express expectation operations via simple linear algebra.

In particular, the \emph{kernel mean embedding}~\citep{Smola07Hilbert} 
of the probability measure $\invmeas$
defined by the Bochner expectation
\begin{equation}
	\label{eq:KME}
	\mu_\invmeas:  = \int_E \varphi(x) \, \dd \invmeas(x) = \Ex[ \varphi(X) ] \in \inrkhs
\end{equation}
naturally satisfies the expectation reproducing property
\begin{equation}
	\label{eq:expectation_reproducing}
	\Ex[ f (X)] 
    =  \Ex\left[ \innerprod{f}{ \varphi(X) }_\inrkhs \right] 
    = \innerprod{f}{ \mu_\invmeas}_\inrkhs \quad \textnormal{for all } f \in \inrkhs.
\end{equation}
We call the RKHS $\inrkhs$ (or equivalently the corresponding kernel $k$)
\emph{characteristic}, if the
\emph{mean embedding map} 
\begin{equation*}
	\pi \mapsto \int_E \varphi(x) \, \dd \invmeas(x) = \mu_\pi \in \inrkhs
\end{equation*}
defined on 
all probability measures on $(E, \sigalg_E)$ is injective. 

\begin{remark}[The RKHS $\inrkhs$ is characteristic]
	\label{rem:characteristic}
	Our Assumptions~\ref{ass:c0} and~\ref{ass:l2} imply
	that $\inrkhs$ is characteristic~\citep{Carmeli10:universality, Sriperumbudur10a,SFL11}.
\end{remark}

For two probability measures $\invmeas, \nu$ on $(E, \sigalg_{E})$,
the so-called \emph{maximum mean discrepancy} (MMD) is defined by
\begin{equation*}
	\mmd(\invmeas, \nu) := 
	\sup_{\substack{f \in \inrkhs \\ \norm{f}_\inrkhs \leq 1}} 
	\abs{ \int_E f(x) \dd \invmeas(x) - \int_E f(x) \dd \nu(x) }
	= \norm{ \mu_\invmeas - \mu_\nu}_\inrkhs.
\end{equation*}
For characteristic kernels, the MMD
constitutes a metric on the set of probability measures on $(E, \sigalg_{E})$. 
This fact has been used as a powerful tool in 
RKHS-based inference~\citep{GrettonTwoSample12, Sejdinovic2013}.

Transferring \eqref{eq:KME} to a regular
conditional distribution of $Y$ given $X$, we define
$\inrkhs$-valued \emph{conditional mean embedding} (CME) function
\citep{Park2020MeasureTheoretic}
\begin{equation*}
	\Fstar(x) := \int_E \varphi(y)  \, \tk(x, \dd y) 
	= \Ex[ \varphi(Y) \mid X = x] \in L^2(E, \sigalg_{E}, \invmeas; \inrkhs)
\end{equation*}
 and obtain a pointwise conditional version of the expectation 
 reproducing property~\eqref{eq:expectation_reproducing} as
\begin{equation}
	\tag{CME}
	\label{eq:CME_property}
	\Ex[ f(Y) \mid  X = x ] = \innerprod{ f }{\Fstar(x)}_{\inrkhs}
	\quad \textnormal{for all } f \in \inrkhs \textnormal{ and } x \in E.
\end{equation}
The fact that $\Fstar$ (or analogously any other regular version
of $\Ex[\varphi(Y) \mid X = \cdot ]$) is a well-defined
element in $L^2(E, \sigalg_{E}, \invmeas; \inrkhs)$
can be seen by using Jensen's inequality for conditional Bochner expectations
as
\begin{align*}
\norm{\Fstar}^2_{L^2(E, \sigalg_{E}, \invmeas; \inrkhs)} &= 
\int_E \norm{ \Fstar(x)  }_\inrkhs^2 \dd \invmeas(x) \\
&\leq \iint_{E \times E } \norm{ \varphi(y)  }_\inrkhs^2  \tk(x, \dd y) \dd \invmeas(x) 
= \Ex[ \norm{\varphi(Y)}_\inrkhs^2 ] < \infty.
\end{align*}
together with Assumption~\ref{ass:second_moment}.

The approximation of $\Fstar$ is a key concept in a wide variety
of models for kernel-based inference. If $\cov[XX]$ is injective,
\citet{SHSF09} and \citet{Fukumizu13:KBR} show that under the assumption
\begin{equation}
	\label{eq:cme_assumption}
	\Ex[ f(Y) \, \mid \, X = \cdot ] = \innerprod{f}{ \Fstar(\cdot)}_\inrkhs   \in \inrkhs
	\quad \textnormal{for all } f \in \inrkhs,
\end{equation}
we have a closed form expression of $ \Fstar$
via 
\begin{equation}
	\label{eq:cme}
	\Fstar(x) = \cov[YX] \cov[XX]^\dagger \varphi(x)
\end{equation} for all $x \in E$ such that $ \varphi(x) \in \range(\cov[XX])$.
Here, the (generally unbounded and not globally defined) operator
$\cov[XX]^\dagger \colon: \range(\cov[XX]) + \range(\cov[XX])^\perp \rightarrow \inrkhs$
is the \emph{Moore--Penrose pseudoinverse} of 
$\cov[XX]$ (see \citet{EHN96}).
The assumption \eqref{eq:cme_assumption} is generally not satisfied as
shown by \citet{Klebanov2019rigorous} in a detailed investigation.
\citet{Gruen12} and \citet{Park2020MeasureTheoretic} show that 
a Tikhonov--Phillips regularized version of the estimate of \eqref{eq:cme}
can be understood as an empirical approximation of $\Fstar$ with functions
in $\vecrkhs$ in a least squares regression context. However, 
no approximation qualities of the CME 
in the $L^2$-operator context are considered.
We will now extend this theory and connect it to the CME
regression model later on.

\section{Nonparametric approximation of  $\ko$}
\label{sec:nonparametric_approximation}
We now restate the main results
from Section~\ref{sec:main_results} with detailed
assumptions and
provide their proofs.
Furthermore, we investigate the connections
of the approximation of $\ko$ over functions in $\inrkhs$
to the maximum mean discrepancy
and regularized least squares regression.

\subsection{Proofs of main results}
 We begin with the proof of
Theorem~\ref{thm:regression_viewpoint},
as it constitutes the theoretical foundation for
our remaining work.
We note that this result can also be interpreted 
as an improvement of a surrogate risk bound derived by
\citet[Section 3.1]{Gruen12} and later on used by \citet{Park2020MeasureTheoretic}
to approximate the CME. We will elaborate on this fact in more
detail later on
(see Section~\ref{sec:least_squares_regression} 
and Remark~\ref{rem:surrogate_risk} in particular). 

\begin{reptheorem}{thm:regression_viewpoint}%
	[Regression and conditional mean approximation]
	Under the Assumptions~\ref{ass:separability}--\ref{ass:second_moment},
	we have for every operator $A \in \HS{\inrkhs}$ that
	\begin{equation*}
		\norm{A - \ko}^2_{\inrkhs \rightarrow L^2(\invmeas)} \leq
		\Ex\left[ \norm{ \Fstar(X) - A^*\varphi(X)}_\inrkhs^2   \right] = 
		\norm{ \Fstar - A^*\varphi(\cdot)  }_{ L^2(E, \sigalg_{E}, \invmeas; \inrkhs) }^2.
	\end{equation*}
	The given bound is sharp.
\end{reptheorem}

\begin{proof}
Let  $A \in \HS{\inrkhs}$. We have
\begin{align*}
	\norm{A - P}^2_{\inrkhs \rightarrow L^2(\invmeas)} &= 
	\sup_{\norm{f}_\inrkhs = 1} \norm{Af - Pf}_{L^2(\invmeas)}^2 \\
	&= \sup_{\norm{f}_\inrkhs = 1} 
	\norm{ [Af](\cdot) - \Ex[ f(Y) \mid X = \cdot ]}_{L^2(\invmeas)}^2 \\
	&= \sup_{\norm{f}_\inrkhs = 1} 
	\norm{  
		\innerprod{Af}{\varphi(\cdot)}_\inrkhs - \innerprod{f}{\Fstar(\cdot)}_\inrkhs
	}_{L^2(\invmeas)}^2 \\
	&= \sup_{\norm{f}_\inrkhs = 1} 
	\norm{  
		\innerprod{f}{ A^*\varphi(\cdot) - \Fstar(\cdot)}_\inrkhs
	}_{L^2(\invmeas)}^2 \\
	&= \sup_{\norm{f}_\inrkhs = 1} 
	\Ex\left[
		\innerprod{f}{ A^*\varphi(X) - \Fstar(X)}_\inrkhs^2
	\right] \\
	&\leq \sup_{\norm{f}_\inrkhs = 1} 
	\Ex\left[
		\norm{f}^2_\inrkhs \norm{ A^*\varphi(X) - \Fstar(X)}_\inrkhs^2
	\right] \\
	&= 
	\Ex\left[
	\norm{ A^*\varphi(X) - \Fstar(X)}_\inrkhs^2
	\right] = \norm{ A^*\varphi(\cdot) - \Fstar  }^2_{L^2(E, \sigalg_{E}, \invmeas; \inrkhs )}, \\
\end{align*}
where we use the reproducing property
in $\inrkhs$ in the third
equality and the Cauchy--Schwarz inequality. It is clear that
the above bound is sharp by considering the case that
we have $\Pr$-a.e.\
$A^*\varphi(X) - \Fstar(X) = h $
for some constant $h \in \inrkhs$.
In this case the above bound
is attained when we choose
$f =  h / \norm{h}_\inrkhs$ in the
supremum.
\end{proof}

\begin{reptheorem}{thm:nonparametric_approximation}%
	[Approximation by Hilbert--Schmidt operators]
	Let Assumptions~\ref{ass:separability}-\ref{ass:l2} be satisfied.
	Then for every $\delta > 0$, 
	there exists a Hilbert--Schmidt operator $A \colon \inrkhs \rightarrow \inrkhs$, such that
	\begin{equation}
		\norm{A - \ko}_{\inrkhs \rightarrow L^2(\invmeas)} < \delta.
	\end{equation}
\end{reptheorem}

\begin{proof}
	By Corollary~\ref{cor:vecrkhs_isomorphism}, every operator
	 $A^* \in \HS{\inrkhs}$ corresponds to a function
	$F \in \vecrkhs$ via $F(x) = A^*\varphi(x)$ for all $x \in E$ and vice versa.	
	The space $\vecrkhs$ is densely embedded into $L^2(E, \sigalg_{E}, \invmeas; \inrkhs)$
	by Remark~\ref{rem:assumptions}(4).
	For every $\delta>0$ we therefore have
	an operator $A^* \in \HS{\inrkhs}$
	such that the bound
	$\norm{A^*\varphi(\cdot) - \Fstar}_{L^2(E, \sigalg_{E}, \invmeas; \inrkhs )}^2 = 
	\norm{F - \Fstar}_{L^2(E, \sigalg_{E}, \invmeas; \inrkhs )}^2 < \delta$
	holds.
	Together with the bound obtained in Theorem~\ref{thm:regression_viewpoint},
	this proves the assertion.
\end{proof}

\begin{repcorollary}{cor:nonparametric_approximation}[]
	Let Assumptions~\ref{ass:separability}-\ref{ass:l2} be satisfied.
	Then there exists a sequence of finite-rank operators $(A_n)_{n \in \N}$
	from $\inrkhs$ to $\inrkhs$ such that 
	$\norm{A_n - \ko}_{\inrkhs \rightarrow L^2(\invmeas)} \to 0$ as $n \to \infty$.
\end{repcorollary}

\begin{proof}
	Let $\delta > 0$. By the fact that
	the finite-rank operators
	on $\inrkhs$ are dense in
	$\HS{\inrkhs}$ and Theorem~\ref{thm:regression_viewpoint},
	we can choose $A \in \HS{\inrkhs}$ as well 
    as a finite-rank operator $A_n$ on $\inrkhs$ such that
	\begin{align*}
		\norm{A_n - \ko}_{\inrkhs \rightarrow L^2(\invmeas)} 
		&\leq \norm{A - \ko}_{\inrkhs \rightarrow L^2(\invmeas)} + \norm{\inrkhsi_\invmeas} \norm{ A_n - A  }_{\inrkhs \rightarrow \inrkhs}
		\\
		&\leq \norm{A - \ko}_{\inrkhs \rightarrow L^2(\invmeas)} + \norm{\inrkhsi_\invmeas} \norm{ A_n - A  }_{\HS{\inrkhs}} < \frac{\delta}{2} + \frac{\delta}{2}.
	\end{align*}
\end{proof}

\subsection{Measure-theoretic implications of the approximation of $\ko$}

When $\inrkhs$ is characteristic,
$\ko: \inrkhs \rightarrow L^2(\invmeas)$ uniquely
determines the conditional distribution $\tk(x, \cdot)$ for 
$\invmeas$-a.e.\ $x \in E$ (that is, up to a choice of a regular version
of the underlying conditional expectation).
This underlines that the conditional expectation operator 
$\ko$ interpreted as an 
operator with the domain $\inrkhs$ instead of $L^2(\nu)$
still captures sufficient information about the underlying joint 
distribution of $X$ and $Y$.
More generally, an approximation
of $\ko$ naturally yields a weighted approximation of 
the associated Markov kernel $\tk$ in the MMD. This
may provide a foundation for the adaptation of MMD-based hypothesis tests
for Markov kernels.

In particular, the following results shows that
$\norm{ \ko - \ko' }_{S_2(\inrkhs, L^2(\invmeas))}^2$
can equivalently be interpreted as the
squared $L^2(E, \sigalg_{E}, \invmeas; \inrkhs)$ distance between
the two conditional mean embeddings 
$\mu_{p(x, \cdot)} = \Fstar(x) = \int_E \varphi(y) \, \tk (x, \dd y)$ and
$\mu_{p'(x, \cdot)} = \Fstar'(x) = \int_E \varphi(y) \, \tk' (x, \dd y)$.

\begin{theorem}[Approximation in MMD]
	\label{thm:approximation_in_mmd}
	Let Assumptions~\ref{ass:separability}--\ref{ass:second_moment}
	be satisfied.
	Let $\ko, \ko': \inrkhs \rightarrow L^2(\invmeas)$ be two
	conditional expectation operators associated with
	the Markov kernels $\tk, \tk' : E \times \sigalg_{E} \rightarrow \R$.
	Then we have
	\begin{equation*}
		\norm{ \ko - \ko' }_{S_2(\inrkhs,L^2(\invmeas))}^2 =
		\int_E \mmd (\tk(x, \cdot), \tk'(x,\cdot) )^2 \, \dd \invmeas(x).
	\end{equation*}
\end{theorem}

\begin{proof}
    Let $(e_i)_{i \in I}$ be a CONS in $\inrkhs$.
    We have
    \begin{align*}
        \norm{ \ko - \ko' }^2_{S_2(\inrkhs, L^2(\invmeas))}
        &= 
        \sum_{i \in I} \norm{ [\ko - \ko'] e_i }^2_{L^2(\invmeas)}
        =
        \sum_{i \in I}
        \int_E \, 
        \left(
        \int_E e_i(y) \, [ \tk - \tk' ](x, \dd y) \right)^2
        \dd \invmeas(x) \\
        &=
        \int_E \, 
        \sum_{i \in I}
        \left(
        \int_E \innerprod{e_i}{\varphi(y)}_\inrkhs \, [ \tk - \tk' ](x, \dd y) \right)^2
        \dd \invmeas(x) \\
        &= 
        \int_E \, 
        \sum_{i \in I}
        \innerprod{e_i}{\mu_{p(x, \cdot)} - \mu_{p'(x, \cdot)}  }^2_\inrkhs 
        \dd \invmeas(x) \\
         &=
		\int_E \mmd (\tk(x, \cdot), \tk'(x,\cdot) )^2 \, \dd \invmeas(x),
    \end{align*}
    where we use the reproducing property in $\inrkhs$ and
    Parseval's identity.
\end{proof}

\begin{remark}[Assumptions of Theorem~\ref{thm:approximation_in_mmd}]
    Note that we do not explicitly assume
	that the underlying random variables 
	associated with $\ko$ and $\ko'$ are distributed with respect to
	the marginals $\invmeas$ and $\nu$. To show the above
	statement, it is sufficient that both
	operators are well-defined and Hilbert--Schmidt
	when the domain and image space and domain are chosen to be
	$\inrkhs$ and $L^2(\invmeas)$ (see Remark~\ref{rem:hs}).
\end{remark}

When $\inrkhs$ is characteristic, we immediately
obtain the following result. It
shows that conditional expectation operators on $\inrkhs$
determine the conditional distribution of the associated random variables
uniquely (up to a choice of a regular version).

\begin{corollary}
	\label{cor:markov_kernel_equivalence}
	Let Assumptions~\ref{ass:separability}--\ref{ass:second_moment}
	be satisfied and $\inrkhs$ be characteristic.
	With the notation of Theorem~\ref{thm:approximation_in_mmd}, we have 
	$\norm{ \ko - \ko' }_{S_2(\inrkhs, L^2(\invmeas))} = 0 $
	if and only if $\tk(x, \cdot) = \tk'(x, \cdot)$ for $\invmeas$-a.e.\ $x \in E$.
\end{corollary}

Corollary~\ref{cor:markov_kernel_equivalence} also implies that
the joint distributions for the class
of pairs of random variables $X,Y$ with a fixed marginal $X \sim \invmeas$
are uniquely determined by $\ko : \inrkhs \rightarrow L^2(\invmeas)$.

\begin{corollary}
	Let $X,X',Y,Y'$ be random variables defined on
	$(\Omega, \sigalg, \Pr)$ taking values in $(E, \sigalg_{E})$
	such that $X \sim \invmeas$ and $X' \sim \invmeas$ and
	Assumptions~\ref{ass:separability}--\ref{ass:second_moment}
	are satisfied for both pairs $X,Y$ and $X',Y'$. Let $\inrkhs$ be 
	characteristic and $\ko, \ko': \inrkhs \rightarrow L^2(\invmeas)$
	be conditional expectation operators given by
	$\ko f = \Ex[ f(Y) \mid X = \cdot]$ and
	$\ko' f = \Ex[ f(Y') \mid X' = \cdot]$ defined by some Markov kernels
	$\tk$ and $\tk'$ respectively.
	Then we have
	$\norm{ \ko - \ko' }_{ S_2(\inrkhs, L^2(\invmeas))} = 0 $
	if and only if $\law(X,Y) = \law(X',Y')$.
\end{corollary}

\begin{proof} Let 
	$\norm{ \ko - \ko' }_{\inrkhs \rightarrow L^2(\invmeas)} = 0 $.
	For any two events $\mathcal{A}, \mathcal{B} \in \sigalg_{E}$,
	we perform the disintegration
	\begin{equation}
		\Pr[X \in \mathcal{A}, Y \in \mathcal{B}]
		= \int_\mathcal{A} p(x, \mathcal{B}) \, \dd \invmeas(x)
	\end{equation}
	and analogously for the pair $X',Y'$. We apply 
	Corollary~\ref{cor:markov_kernel_equivalence},
	leading to the $\invmeas$-a.e.\ equivalence
	$\tk(\cdot, \mathcal{B}) = \tk'(\cdot, \mathcal{B})$.
	This gives
	$\Pr[X \in \mathcal{A}, Y \in \mathcal{B}]
	= \Pr[X' \in \mathcal{A}, Y' \in \mathcal{B}]$.
	The converse implication follows analogously.
\end{proof}

\subsection{Least squares regression and connection to the CME}
\label{sec:least_squares_regression}

We now describe the theoretical foundation of estimating 
$\ko$ based on Theorem~\ref{thm:regression_viewpoint}.
In the process, we will see that our concept is closely related to the CME.

By the operator reproducing
property from Corollary~\ref{cor:vecrkhs_isomorphism} 
we may rewrite the
vRKHS least squares regression problem
\begin{equation}
	\label{eq:function_regression_problem}
	\argmin_{F \in \vecrkhs} \risk(F)
	\textnormal{ with } 
	\risk(F) := \Ex[ \norm{\varphi(Y) - F(X) }_\inrkhs^2 ]
\end{equation}
equivalently as
\begin{equation}
	\label{eq:operator_regression_problem}
	\argmin_{A^* \in \HS{\inrkhs}} \Ex[ \norm{\varphi(Y) - A^*\varphi(X) }_\inrkhs^2 ].
\end{equation}
As is well-known in statistical learning 
theory~\citep[Proposition 1]{Cucker02:Theory},
for all $F \in L^2(E, \sigalg, \invmeas; \inrkhs)$,
the risk $\risk$ allows for the decomposition
 \begin{equation}
 	\label{eq:risk_decomposition}
 	\risk(F) = \norm{\Fstar - F}^2_{L^2(E, \sigalg, \invmeas; \inrkhs)} 
 	+ \risk(\Fstar),
 \end{equation}
where $\risk(\Fstar)$ represents the irreducible error term
(see Theorem~\ref{thm:risk_reformulation}
for a proof in the infinite-dimensional case). 
This reduces the regression 
problem~\eqref{eq:function_regression_problem} and equivalently 
problem~\eqref{eq:operator_regression_problem}
to an $L^2$-approximation of the conditional mean embedding $\Fstar$.
In this context, $\Fstar$ is often 
called \emph{regression function}.
Therefore, the so-called \emph{excess risk} 
$\risk(F) - \risk(\Fstar) = 
\norm{\Fstar - F}^2_{L^2(E, \sigalg, \invmeas; \inrkhs)} $
of some estimate $F \in \vecrkhs$ is typically investigated in
nonparametric statistics.

The above formalism allows us to estimate 
the conditional mean operator $\ko$ based on our previous
results.
By Theorem~\ref{thm:regression_viewpoint},
we have
\begin{equation}
	\norm{A - \ko}^2_{\inrkhs \rightarrow L^2(\invmeas)} \leq
	\norm{ \Fstar - A^*\varphi(\cdot)  }_{ L^2(E, \sigalg_{E}, \invmeas; \inrkhs) }^2
\end{equation}
for all $A^* \in \HS{\inrkhs}$.
We can now perform 
the vRKHS regression~\eqref{eq:operator_regression_problem}
and obtain an approximation of $\ko$ 
in the norm $	\norm{\cdot}^2_{\inrkhs \rightarrow L^2(\invmeas)} $ in terms
of $A \in \HS{\inrkhs}$, which we implicitly interpret
as an operator from $\inrkhs $ to $L^2(\invmeas)$.
Theorem~\ref{thm:nonparametric_approximation}
and Corollary~\ref{cor:nonparametric_approximation} show
that this is possible up to an arbitrary degree of accuracy.

Along the lines of the known work on least squares regression of
the form~\eqref{eq:function_regression_problem} or
equivalently~\eqref{eq:operator_regression_problem},
we can distinguish following two general cases:
\begin{enumerate}
	\item The \emph{well-specified case}, i.e., there exists
	a regular version of the conditional distribution of
	$Y$ given $X$ such that 
	$\Fstar(\cdot) = \Ex[ \varphi(Y) \mid X = \cdot ] \in \vecrkhs$.
	For the well-specified case, we below obtain the known properties
	of the conditional mean embedding which were derived
	from the linear-algebraic 
	perspective~\citep{SHSF09,Klebanov2019rigorous,Klebanov2020linear}.
	\item The \emph{misspecified case}, i.e., 
	$\Fstar \in L^2(\invmeas) \setminus \vecrkhs$. This is clearly
	the more interesting setting, as the well-specified case is
	typically not ensured in practice. From
	the operator-theoretic perspective, 
	this case has not been investigated yet.
\end{enumerate}

Our previous results allows reformulation of the well-specified case
and establishes a connection to the CME.

\begin{corollary}[Well-specified case]
	\label{cor:well-specified-case}
	Let Assumption~\ref{ass:separability}--\ref{ass:second_moment}
	be satisfied. 
	Consider a fixed regular version of the distribution of $Y$ conditioned on $X$
	given by some Markov kernel $\tk: E \times \sigalg_{E} \rightarrow \R$.
	The following statements are equivalent:
	\begin{enumerate}[label=(\roman*)]
	
	\item We have
		$\Fstar(\cdot) = \Ex[ \varphi(Y) \mid X = \cdot ] \in \vecrkhs$.

	\item There exists an operator $A \in \HS{\inrkhs}$
	such that
	\begin{equation}
		\label{eq:well-specified}
		[Af](x) = \innerprod{Af}{\varphi(x)}_\inrkhs = 
		\innerprod{f}{A^*\varphi(x)}_\inrkhs = \Ex[f(Y) \mid X = x]
	\end{equation}
	for all $x \in E$ and $f \in \inrkhs$.
	\end{enumerate}
	
	Both (i) and (ii) imply (iii):
	
	\begin{enumerate}[resume, label=(\roman*)]
	
	\item There exists an operator $A \in \HS{\inrkhs}$ which 
	satisfies
	$\norm{A-\ko}_{\inrkhs \rightarrow L^2(\invmeas)} = 0$.
	
	\end{enumerate}

\end{corollary}

\begin{proof} We show that (i) is equal to (ii).
	Let $\Fstar(\cdot) = \Ex[ \varphi(Y) \mid X = \cdot ] \in \vecrkhs$.
	Let $A^* \in \HS{\inrkhs}$ be the unique operator such that
	$A^*\varphi(\cdot) = \Fstar(\cdot)$ by
	Corollary~\ref{cor:vecrkhs_isomorphism}. 
	By the reproducing property in $\inrkhs$,
	we can verify \eqref{eq:well-specified} immediately.
	For the converse implication, let \eqref{eq:well-specified}
	be satisfied
	for some operator $A^* \in \HS{\inrkhs}$.
	Then by Corollary~\ref{cor:vecrkhs_isomorphism},
	we have the function $F \in \vecrkhs$ with
	$F(\cdot) = A^* \varphi(\cdot)$
	such that 
	\begin{equation}
		\label{eq:well_specified_case_proof}
		\innerprod{f}{F(x)}_\inrkhs = \Ex[f(Y) \mid X = x]
		= \Ex[\innerprod{f}{\varphi(Y)}_\inrkhs \mid X = x]
	\end{equation}
	for all $f \in \inrkhs$.
	The right hand side of~\ref{eq:well_specified_case_proof}
	is equal to $\innerprod{f}{\Ex[{\varphi(Y)}_\inrkhs \mid X = x]}_\inrkhs$
	for all $x \in E$ and $f \in \inrkhs$,
	we therefore have 
	$F(\cdot) 
	= \Ex[ \varphi(Y) \mid X = \cdot ] = \Fstar(\cdot) \in \vecrkhs$ 
	as claimed.
	The last statement follows from
	Theorem~\ref{thm:regression_viewpoint} by inserting
	$A^*$ into the right hand side of the bound,
	giving $\norm{A- \ko}_{\inrkhs \rightarrow L^2(\invmeas)} = 0$.
\end{proof}

\begin{remark}[Connection to CME and well-specified case]
	By comparing~\eqref{eq:well-specified} to the 
	expectation reproducing property~\eqref{eq:CME_property},
	we see that in the well-specified case,
	the operator $A^*$ satisfying~\eqref{eq:well-specified}
	is exactly the operator which was introduced
	by \citet{SHSF09} as the original conditional mean embedding.
	\emph{That is, we obtain the approximation of
	$\ko$ from $\inrkhs$ to
	$L^2(\invmeas)$ as the adjoint of the CME}.
	A similar connection was established by~\citet{Klus2019}
	under the restrictive assumptions of~\cite{SHSF09} in the context of
	Markov operators.	
\end{remark}

\begin{remark}[Well-specified case closed form solution]
	\citet[Theorem 5.8]{Klebanov2020linear} 
	prove in a slightly different context of tensor product spaces
	without explicitly using vRKHSs, that in the well-specified case
	the operator
	$A^*$ satisfying~\eqref{eq:well-specified} 
	can be expressed in terms of the covariance
	operators as $A^* = (\cov[XX]^\dagger \cov[XY])^*$.
	In fact, this proves that $(\cov[XX]^\dagger \cov[XY])^*$
	is Hilbert--Schmidt in this case.
\end{remark}

\begin{remark}[Surrogate risk bound for the CME]
	\label{rem:surrogate_risk}
	In the well-specified case, \citet{Park2020MeasureTheoretic}
	investigate the estimation of the 
	CME in terms
	of~\eqref{eq:function_regression_problem}.
	Their results build upon the surrogate risk
	bound 
	\begin{equation*}
		\norm{A - \ko}^2_{\inrkhs \rightarrow L^2(\invmeas)}
		\leq \risk(A^*\varphi(\cdot)),
	\end{equation*} originally formulated by \citet{Gruen12}.
	Our Theorem~\ref{thm:regression_viewpoint} improves this
	bound and eliminates the need for additional approximation results
	~\citep[Theorem 3.2]{Gruen12} 
	for the analysis of the misspecified case.
	By \eqref{eq:risk_decomposition}, our bound 
	from Theorem~\ref{thm:regression_viewpoint} equals to
	\begin{equation*}
		\norm{A - \ko}^2_{\inrkhs \rightarrow L^2(\invmeas)}
		\leq \risk(A^*\varphi(\cdot)) - \risk({\Fstar}),
	\end{equation*} 
	 which allows the approximation up to an arbitrary accuracy and
	removes the excess term $\risk(\Fstar)$. 
\end{remark}

We have seen that in the well-specified case, our results align
with prior work on the CME.
In the practically more relevant 
misspecified case however, the bound given by 
Theorem~\ref{thm:regression_viewpoint} significantly simplifies
the theory of approximating the CME.
For the remainder of the paper, we will focus on
the empirical estimation of $\ko$
without restricting ourselves to the well-specified case.

\section{Empirical estimation and regularization theory}
\label{sec:empirical_estimation}

We now connect our previous results to the
theory of supervised learning and derive
empirical estimators of $\ko$.
To this end, we will briefly review
how the regression problem \eqref{eq:function_regression_problem} 
can be formulated
in terms of an inverse problem.
The decomposition of $\risk$ in
\eqref{eq:risk_decomposition} allows to obtain a solution
by approximating $\Fstar$ with functions
in $\vecrkhs$. This framework allows to 
derive the well-known formalism for supervised learning
and regularization theory
which will yield estimates of $\ko$.
We refer
to the seminal work for least squares regression
with vRKHSs \citep{Caponnetto2007} for
more details. This section contains the reformulation of our
setting in terms of known results,
making the theory of vRKHS regression applicable
for the estimation of $\ko$. 
We use this framework to
derive new results in Section~\ref{sec:tikhonov_phillips_regularization}.

\subsection{Inverse problem}
\label{sec:inverse_problem}

In the misspecified case,
it is not necessarily clear that the minimizer 
of $\risk$ over $\vecrkhs$ exists.
The analytical nature of this question
can be naturally expressed in terms of an 
inverse problem. 
For the necessary background on inverse problems in Hilbert spaces and regularization theory,
we refer to~\citet{EHN96}.
We will formulate~\eqref{eq:function_regression_problem}
a bit more verbosely in terms of the inclusion
$\vecrkhsi_\invmeas: \inrkhs \rightarrow L^2(E, \sigalg_E, \invmeas; \inrkhs)$,
so that the connection to the inverse problem becomes clear.

If $F \in \vecrkhs$, we have by~\eqref{eq:risk_decomposition} that
\begin{equation*}
	\risk(F) = \norm{ \vecrkhsi_\invmeas F - \Fstar }^2_{L^2(E, \sigalg_E, \invmeas; \inrkhs)} + \risk(\Fstar).
\end{equation*}
Finding $\FG := \argmin_{F \in \vecrkhs} \risk(F)$
is therefore equivalent to finding $\FG \in \vecrkhs$ such that
\begin{equation*}
	\norm{ \vecrkhsi_\invmeas F - \Fstar }^2_{L^2(E, \sigalg_E, \invmeas; \inrkhs)}
\end{equation*}
is minimal.
As is well-known from the theory of inverse problems,
this is equivalent to finding
the optimal solution $\FG$
of the potentially ill-posed inverse problem
\begin{equation}
	\label{eq:inverse_problem}
	\vecrkhsi_\invmeas F = \Fstar, \quad F \in \vecrkhs.
\end{equation}
The inverse problem \eqref{eq:inverse_problem}
is again equivalent to finding the
solution of the so-called \emph{normal equation}
\citep[Theorem 2.6]{EHN96} given by
\begin{equation*}
	(\vecrkhsi_\invmeas^* \vecrkhsi_\invmeas ) 
	F = \vecrkhscov F = \vecrkhsi_\invmeas^* \Fstar, \quad F \in \vecrkhs.
\end{equation*}
In particular, we obtain the following solution.

\begin{theorem}[Regression solution]
	\label{thm:surrogate_inverse}
	Let Assumptions~\ref{ass:separability}--\ref{ass:second_moment}
	be satisfied.
	The optimal solution
	\begin{equation*}
		\FG = \argmin_{F \in \vecrkhs} R(F) = \argmin_{F \in \vecrkhs}
		\norm{ \vecrkhsi_\invmeas F - \Fstar }^2_{L^2(E, \sigalg_E, \invmeas; \inrkhs)}
	\end{equation*}
	exists if and only if 
	$\vecrkhsi_\invmeas^* \Fstar \in \range(\vecrkhscov) + \range(\vecrkhscov)^\perp =: \domain(T^\dagger)$,\!\footnote{An
		equivalent condition is
		$ \proj \Fstar \in \range(\vecrkhsi_\invmeas)$,
		where $\proj \colon L^2(E, \sigalg_E, \invmeas; \inrkhs) 
		\rightarrow L^2(E, \sigalg_E, \invmeas; \inrkhs)$ is the 
		orthogonal projection onto the closure of 
		$\range(\vecrkhsi_\invmeas)$.
	}  	
	where the operator $\vecrkhscov^\dagger : \range(\vecrkhscov) + \range(\vecrkhscov)^\perp \rightarrow \vecrkhs$
	is the pseudoinverse
	of $\vecrkhscov$.
	In this case,
	$ \FG$ is given by the solution to the
	normal equation
	\begin{equation}
		\label{eq:normal_equation}
		\vecrkhscov F = \vecrkhsi_\invmeas^* \Fstar, \quad F \in \vecrkhs
	\end{equation}
	in terms of $\FG = \vecrkhscov^\dagger \vecrkhsi_\invmeas^* \Fstar$.
\end{theorem}

\subsection{Regularization and empirical estimation}

For simplicity, we assume
that the optimal solution $\FG = \argmin_{\vecrkhs} \risk(F)$
exists, i.e., we have $ \vecrkhsi_\invmeas^* \Fstar \in \domain(\vecrkhscov^\dagger)$.
We wish to compute
a solution of the normal equation
\begin{equation}
	\label{eq:normal_equation1}
	\vecrkhscov F = \vecrkhsi_\invmeas^* \Fstar, \quad F \in \vecrkhs.
\end{equation}
in terms of $\FG = \vecrkhscov^\dagger \vecrkhsi_\invmeas^* \Fstar$
based on an empirical realization of $(X_t)_{t \in \Z}$.

In order to do this, we must discretize $\vecrkhscov$ 
as well as the right-hand side $ \vecrkhsi_\invmeas^* \Fstar$.
We now face the problem that
\eqref{eq:normal_equation1} 
may be \emph{ill-posed} in the sense that
the solution does not continously depend on $\vecrkhsi_\invmeas^* \Fstar$ 
(and of course on $\vecrkhscov$ as well).
To still be able to perform an estimation, 
a \emph{regularization strategy}~\citep{EHN96}
is needed to ensure well-posedness in practice.

Let 
$\{ g_\reg( \vecrkhscov): \vecrkhs \rightarrow \vecrkhs \, | \, \reg \in (0, \infty ] \} $
be a regularization strategy.\!\footnote{
	We require 
	$\{ g_\reg( \vecrkhscov): \vecrkhs \rightarrow \vecrkhs \, | \, \reg \in (0, \infty ] \} $
	to be a parametrized family of \emph{globally defined bounded operators}
	satisfying $ g_\reg(\vecrkhscov) F \to \vecrkhscov^\dagger F$
	for all $F \in \domain(\vecrkhscov^\dagger$)
	as $\reg \to 0$.
}
For a fixed regularization parameter $ \reg >0$, we define
the regularized solution
\begin{equation}
	\label{eq:regF}
	\regF := g_\reg(\vecrkhscov) \vecrkhsi_\invmeas^* \Fstar \in \vecrkhs.
\end{equation}

We now discretize the regularized problem~\eqref{eq:regF} 
based on the empirical data
\[
	\bz := ((X_1, Y_1), \dots, (X_n, Y_n)),
\]
where we assume iid $(X_i , Y_i ) \sim \law(X,Y)$.
We generalize the \emph{sampling operator approach}
\citep{Smale2005} from the scalar setting to
the vector-valued scenario and derive an empirical
estimate of $\regF$.
Given the data above, we define the \emph{sampling operator}
$ \samplingop \colon \vecrkhs \rightarrow \inrkhs^n$
given by
$
	\samplingop F := ( F(X_t) )_{t = 1}^n 
	= ( K_{X_t}^* F )_{t = 1}^n.
$
Here, we consider $\inrkhs^n$ as a Hilbert space
equipped with the inner product
\begin{equation*}
	\innerprod{\mathbf{f}}{\mathbf{h}}_{\inrkhs^n} := \frac{1}{n}
	\sum_{i=1}^n \innerprod{ f_i } {h_i}_\inrkhs
\end{equation*}
for $\mathbf{f} = (f_1, \dots, f_n) \in \inrkhs^n$
and $\mathbf{h} = (h_1, \dots, h_n) \in \inrkhs^n$.
It is easy to see that the adjoint of $\samplingop$ is
the operator $\samplingop^* \colon \inrkhs^n \rightarrow \vecrkhs $
given by 
\begin{equation*}
	\samplingop^* \mathbf{h} = 
	\frac{1}{n} \sum_{i = 1}^n K_{X_i} h_i
\end{equation*}
for all $\mathbf{h} \in \inrkhs^n$
and the operator 
$\empvecrkhscov := 
\samplingop^* \samplingop \colon \vecrkhs \rightarrow \vecrkhs$
satisfies
\begin{equation*}
	\empvecrkhscov F = \samplingop^* \samplingop F = \frac{1}{n} 
	\sum_{i = 1}^n K_{X_i} K_{X_i}^* F
\end{equation*}
for all $F \in \vecrkhs$.
Based on these considerations, we will use
$\samplingop^*$ and $\empvecrkhscov$
as empirical estimates for 
$\vecrkhsi_\invmeas^*$ and $\vecrkhscov$ respectively based on 
the data $\bx$.
We define the \emph{target data vector}
	$\Upsilon := ( \varphi( Y_1 ), \dots, \varphi(Y_n) ) \in \inrkhs^n$
and obtain the empirical regularized solution
\begin{equation}
	\label{eq:empreg}
	\empregF := g_\reg(\empvecrkhscov) \samplingop^* \Upsilon \in \vecrkhs
\end{equation}
as the discretized analogue of the 
analytical regularized solution \eqref{eq:regF}.

Via the identification
of $\regF$ and $\empregF$ with
operators through the isomorphism $\Theta$ in 
Corollary~\ref{cor:vecrkhs_isomorphism},
we obtain the
\emph{analytical regularized operator solution}
\begin{equation*}
	\regA := [\Theta^{-1}(\regF)]^* \in \HS{\inrkhs}
\end{equation*}
as well as the \emph{empirical regularized operator solution}
\begin{equation*}
	\empregA := [\Theta^{-1}(\empregF)]^* \in \HS{\inrkhs}, 
\end{equation*} i.e.,
$\regF(x) = \regA\varphi(x)$ and $\empregF(x) = \empregA\varphi(x)$
for all $x \in E$.

\begin{remark}[Convergence rates $\empregF \to \Fstar$]
    \citet{MollenhauerEtAl2022} provide a convergence analysis
    for regularized least squares regression with infinite-dimensional 
    outputs which cover the estimate $\empregF$ constructed in
    this section.
    Under classical smoothness assumptions for
    $\Fstar$ in the well-specified case, probabilistic rates up to $1/\sqrt{n}$
    are obtained for generic regularisation schemes. 
    These rates match known lower bounds
    on rates for classical kernel regression with scalar response under
    analogous assumptions.
    Moreover, \citet{LiEtAl2022} prove optimal rates for the misspecified case
    in terms of norms in interpolation spaces between
    $\vecrkhs$ and $L^2(E, \sigalg:E, \invmeas; \inrkhs)$ for the 
    special Tikhonov--Phillips regularization case 
    (see Section~\ref{sec:tikhonov_phillips_regularization}).
    These results cover fast rates up to $1/n$ for more sophisticated assumptions
    about the underlying joint distribution of $X$ and $Y$.
\end{remark}

\section{Tikhonov-Phillips regularization}
\label{sec:tikhonov_phillips_regularization}
For the remainder of this paper, we will restrict ourselves
to the Tikhonov--Phillips regularization approach \citep{TA77}
to solve the (potentially ill-posed) inverse problem
given by Theorem~\ref{thm:surrogate_inverse}
in order to obtain the optimal solution $\FG$ in $\vecrkhs$
of the surrogate problem (assuming it exists).

\subsection{General framework}
Tikhonov--Phillips regularization corresponds
to the regularization strategy
$ g_\reg(\vecrkhscov) := (\vecrkhscov + \reg \id_\vecrkhs)^{-1} \in \bounded{\vecrkhs}$ for 
$\reg > 0$.
We replace the risk $\risk$ with the
\emph{regularized risk} 
\begin{equation}
	\label{eq:regrisk}
	\regrisk(F) := \risk(F) + \reg \norm{F}_\vecrkhs^2
\end{equation}
with a regularization parameter $\reg > 0 $.
The unique minimizer of~\eqref{eq:regrisk}
exists for all $\reg > 0$ and is exactly given by the regularized solution
$
	\regF = (\vecrkhscov + \reg \id_\vecrkhs)^{-1} \vecrkhsi_\invmeas^*\Fstar
$,
which is a standard result
in inverse problems~\citep[Theorem 5.1]{EHN96}.
Based on the data
$\bz$, we define the \emph{regularized empirical risk}
\begin{equation}
	\label{eq:empregrisk}
	\empregrisk(F) := 
	\frac{1}{n}\sum_{i = 1}^n \norm{ \varphi(Y_i) - F(X_i) }_\inrkhs^2
	+ \reg \norm{F}^2_\vecrkhs
\end{equation}
for all $F \in \vecrkhs$.
We can reformulate \eqref{eq:empregrisk} in terms
of the sampling operator equivalently as
$
	\empregrisk(F) = \norm{ \samplingop F - \Upsilon }^2_{\inrkhs^n} 
	+ \reg \norm{ F }^2_\vecrkhs
$
for all $F \in \vecrkhs$.
Therefore, $\empregrisk$ admits a unique minimizer
in $\vecrkhs$ given by the regularized empirical solution
$
	\empregF = (\empvecrkhscov + \reg \id_\vecrkhs)^{-1} \samplingop^* \Upsilon,
$
which we will consider from now on as the estimate
of~$\regF$.

\subsection{Closed form Tikhonov--Phillips operator estimates}
\label{sec:recovering_operator}

We show that for the Tikhonov--Phillips estimate,
the adjoint of the regularized analytical operator solution
$\regA^* = \Theta^{-1}(\regF) $ which satisfies
\begin{equation*}
	 \regA^*  = \argmin_{A \in \HS{\inrkhs}  } 
	\Ex[ \norm{   \varphi(Y) - A^*\varphi(X) }^2_\inrkhs  ] + \reg \norm{A}^2_{\HS{\inrkhs}}
\end{equation*}
admits a closed form
representation in terms of covariance operators associated with the kernel $k$.
In fact, we prove that $\regA^*$ has the known form
which \citet{SHSF09} originally identified as the conditional mean embedding
under the previously mentioned restrictive assumptions.


While this result does not come as a surprise at this point, 
we emphasize that this has not been
proven before. Although \citet{Gruen12} establish 
a connection between the \emph{empirical} regularized solution $\empregF$
and a version of the \emph{empirical} conditional mean embedding
with a \emph{rescaled regularization parameter}, a
population analogue was never derived. A simple asymptotic argument
via convergence in the infinite-data limit 
is hampered by the rescaling of the regularization
parameter in this derivation. 
Interestingly, the population expression of $\regA$ 
which we derive here is sometimes
taken for granted in the literature (see for example \citet{Fukumizu13:KBR}),
even if it was never proven in the original work.

Our analysis offers a view on the beautiful duality between
the generalized covariance operator $\vecrkhscov$ acting on $\vecrkhs$,
composition operators acting on $\HS{\inrkhs}$ and the 
kernel covariance operator $\cov[XX]$.

\begin{remark}
	While our analysis is purely aimed at a theoretical understanding at this point,
	we expect that the following results will have a practical
	benefit, as they allow an asymptotic discussion
	of the spectral properties of the given estimates 
	(see also Section~\ref{sec:markov_operators}).
\end{remark}

For an operator $B \in \bounded{\inrkhs}$,
define the \emph{right-composition operator} 
\begin{align}
	\comp_B \colon \HS{\inrkhs} &\rightarrow \HS{\inrkhs}, \\
	A &\mapsto AB.
\end{align}
It is easy to see that $\comp_B$ is a
well-defined bounded operator since $\HS{\inrkhs}$ is an ideal in 
$\bounded{\inrkhs}$ and we have 
$\norm{\comp_BA}_{\HS{\inrkhs}} \leq \norm{A}_{\HS{\inrkhs}} \norm{B}$.
Furthermore, if $B$ is invertible then $\comp_B$ is invertible and we have
$\comp_{B^{-1}} = \comp^{-1}_B$.

The following result describes the connection between $\vecrkhs$ and $\cov[XX]$
in terms of the composition operator $\comp_{\cov[XX]}$. In fact,
it shows that $\vecrkhscov \colon \vecrkhs \rightarrow \vecrkhs$ 
describes exactly the action of $\comp_{\cov[XX]} \colon \HS{\inrkhs}  \rightarrow\HS{\inrkhs} $ 
under the isomorphism $\Theta \colon \HS{\inrkhs} \rightarrow \vecrkhs$.

%

\begin{figure}[htb]
	\centering
	\begin{subfigure}{.49\textwidth}
		\centering
		\begin{tikzcd}[column sep=large, row sep =large]
			\vecrkhs \arrow[d, "\vecrkhscov"', shift right] \arrow[r, "\Theta^{-1}"] & \HS{\inrkhs} 
			\arrow[d, "\comp_{ \cov[XX] }"] \\
			 \vecrkhs \arrow[r, "\Theta^{-1}"']                            & \HS{\inrkhs}             
		\end{tikzcd}
	\end{subfigure}
	\vspace{1em}
	\begin{subfigure}{.49\textwidth}
		\centering
		\begin{tikzcd}[column sep=large, row sep =large]
			\vecrkhs \arrow[d, "\vecrkhscov + \reg \id_\vecrkhs"', shift right] 
			\arrow[r, "\Theta^{-1}"] & \HS{\inrkhs}  \arrow[d, "\comp_{ \cov[XX] + \reg \id_{\inrkhs} }"] \\
			\vecrkhs \arrow[r, "\Theta^{-1}"']                            & \HS{\inrkhs}               
		\end{tikzcd}
	\end{subfigure}	
	
	\caption{Correspondence of $\vecrkhscov$ 
		and $\comp_{\cov[XX]}$.} \label{fig:reg_vecrkhscov}
\end{figure} 

\begin{theorem}
	\label{thm:comp_operator}
	Let $F \in \vecrkhs$ 
	and $A := \Theta^{-1}(F) \in \HS{\inrkhs}$.
	Then the diagrams in Figure \ref{fig:reg_vecrkhscov}
	are both commutative diagrams, i.e., we have
	\begin{equation*}
		\Theta^{-1}(\vecrkhscov F) = 
		A \cov[XX]
	\end{equation*}
	as well as
	\begin{equation*}
		\Theta^{-1}[(\vecrkhscov + \reg  \id_\vecrkhs) F] =
		A (\cov[XX] + \reg \id_\inrkhs).
	\end{equation*}
\end{theorem}

\begin{proof}
	Let $F \in \vecrkhs$ and $A = \Theta^{-1} (F) \in \HS{\inrkhs}$. We have 
	$F(\cdot) = A\varphi(\cdot)$ by Corollary \ref{cor:vecrkhs_isomorphism}.
	From the definition of $\vecrkhs$, we get
	\begin{align*}
		T F &= \int_E K_x F(x) \dd \invmeas(x) = \int_E K_x [ A \varphi(x) ] \dd \invmeas(x) \\
		&=  \int_E   A [ k(\cdot, x ) \varphi(x) ] \dd \invmeas(x) 
		= A \int_E   k(\cdot, x ) \varphi(x)  \dd \invmeas(x) \\
		&= A  \int_E   [\varphi(x) \otimes \varphi(x)] \varphi(\cdot) \dd \invmeas(x) 
		= A \cov[XX] \varphi(\cdot),
	\end{align*}
	where we use the fact that for every fixed $x' \in E$, 
	the map $x \mapsto k(x',x) \varphi(x)$ is an element of  
	$L^1(E, \sigalg_{E}, \invmeas; \inrkhs) $ due to Assumption \ref{ass:second_moment} and
	Hölder's inequality. Because of this,
	the integration and the operator $A$ commute \citep[Chapter II.2, Theorem 6]{DiestelUhl1977}.
	The operator $A \cov[XX]$ is Hilbert--Schmidt and $TF = A \cov[XX] \varphi(\cdot)$
	confirms the operator reproducing property under $\Theta^{-1}$ from
	Corollary~\ref{cor:vecrkhs_isomorphism}, hence we have $\Theta^{-1}(\vecrkhscov F) = A \cov[XX]$.
	Using this fact, we obtain
	\begin{align*}
		(\vecrkhscov + \reg \id_\vecrkhs)F =  A \cov[XX] \varphi(\cdot) +  \reg A \varphi(\cdot) 
		= A(\cov[XX] + \reg \id_\inrkhs) \varphi(\cdot),
	\end{align*} confirming the same relation for the second assertion of the theorem.
\end{proof}

Theorem~\ref{thm:comp_operator} allows us to easily derive
the expression for the
Tikhonov--Phillips estimate $\regF$ under $\Theta^{-1}$ 
in terms of its corresponding operator in $\HS{\inrkhs}$
in terms of $\cov[XX]$ and $\cov[YX]$.

\begin{corollary}[Closed form analytical operator solution]
	\label{cor:regularized_operator_solution}
	We have 
	\begin{equation*}
		\Theta^{-1}(\regF) = \regA^* = \cov[YX]  ( \cov[XX] + \reg \id_\inrkhs)^{-1},
	\end{equation*}
	i.e., 
	the analytical regularized operator solution can be represented as
	\begin{equation}
		\Theta^{-1}(\regF)^*  = \regA = ( \cov[XX] + \reg \id_\inrkhs)^{-1} \cov[XY].
	\end{equation}
\end{corollary}

\begin{proof}
	By definition, we have
	$
	\regF = g_\reg(\vecrkhscov) \vecrkhsi_\invmeas^* \Fstar 
	= (\vecrkhscov + \reg \id_\vecrkhs)^{-1} \vecrkhsi_\invmeas^* \Fstar$.
	We can rearrange
	\begin{align*}
		\vecrkhsi_\invmeas^* \Fstar &= \int_E K(\cdot, x) \Fstar(x) \dd \invmeas(x) 
		= \int_E k(\cdot, x) \int_E \varphi(y) p(x, \dd y) \dd \invmeas(x) \\
		&=\iint_{ E^2 } \varphi(y) \innerprod{\varphi(x)}{\varphi(\cdot)}_\inrkhs 
		p(x, \dd y) \dd \invmeas(x) \\
		&= \left[ \int_E \varphi(Y) \otimes \varphi(X) \dd \Pr  \right] \varphi(\cdot)
		=\cov[YX] \varphi(\cdot).
	\end{align*}
	We have thus shown that
	$ \cov[YX] = \Theta^{-1} (\vecrkhsi_\invmeas^* \Fstar)$ 
	by the operator reproducing property 
	from Corollary~\ref{cor:vecrkhs_isomorphism}.
	Theorem~\ref{thm:comp_operator} implies that
	the operator $(\vecrkhscov + \reg \id_\vecrkhs)^{-1}$ acting on $\vecrkhs$ 
	may be represented under $\Theta^{-1}$ 
	as by the right composition operator 
	$\comp_{(\cov[XX] + \reg \id_\inrkhs)^{-1}}$ acting on on $\HS{\inrkhs}$, 
	leading to
	\begin{equation*}
		\Theta^{-1}(\regF) = 
		\comp_{(\cov[XX] + \reg \id_\inrkhs)^{-1}} \cov[YX]
		= \cov[YX](\cov[XX] + \reg \id_\inrkhs)^{-1}
	\end{equation*}
as claimed.
\end{proof}

Analogously we obtain a closed form representation
for the empirical regularized operator solution
$\empregA$, in terms of the emprical covariance operators 
\begin{equation*}
	\empcov[XX] := \frac{1}{n} \sum_{i=1}^n \varphi(X_i) \otimes \varphi(X_i)
	\textnormal{ and } \empcov[XY] :=
	\frac{1}{n} \sum_{i=1}^n \varphi(Y_i) \otimes \varphi(X_i).
\end{equation*}

\begin{theorem}[Closed form empirical operator solution]
	\label{thm:empirical_regularized_operator_solution}
	We have 
	\begin{equation*}
		\Theta^{-1}(\empregF) = \empregA^* 
		= \empcov[YX]  ( \empcov[XX] + \reg \id_\inrkhs)^{-1},
	\end{equation*}
	i.e., 
	the empirical regularized operator solution can be represented as
	\begin{equation}
		\Theta^{-1}(\empregF)^*  = \empregA = ( \empcov[XX] + \reg \id_\inrkhs)^{-1} \empcov[XY].
	\end{equation}
\end{theorem}

Theorem~\ref{thm:empirical_regularized_operator_solution} can
be proven by simply replacing $\vecrkhscov$ with 
the sample-based operator $\empvecrkhscov$ in the proof of
Theorem~\ref{thm:comp_operator}, leading
to 
$
\Theta^{-1}[(\empvecrkhscov + \reg  \id_\vecrkhs) F] =
\Theta^{-1}(F) (\empcov[XX] + \reg \id_\inrkhs)$ for all 
$F \in \vecrkhs$.
Furthermore replacing
$\vecrkhsi_\invmeas^*$ with $\samplingop^*$ in the proof of
Corollary~\ref{cor:regularized_operator_solution}
yields $\Theta^{-1}(\samplingop^* \Upsilon) = \empcov[YX]$, 
thereby confirming 
the claim when applying both results to
$\empregA = \Theta^{-1}(\empregF ) = 
\Theta^{-1}[(\empvecrkhscov + \reg \id_\vecrkhs)^{-1} \samplingop^* \Upsilon]$.

%
%
%
%

\section{Application: kernel-based extended dynamic mode decomposition}
\label{sec:markov_operators}
	The derivation of the closed form for the regularized operator solution
	from the previous section allows to connect
	our theory to known spectral analysis techniques used in practice.
	
	\citet{Klus2019}
	and \citet{MollenhauerSVD2020} show that the eigenfunctions of
	the regularized empirical estimate 
	$\empregA = ( \empcov[XX] + \reg \id_\inrkhs)^{-1} \empcov[XY]$
	can be computed by solving a matrix eigenproblem.
	In the case
	that $\ko$ is
	the Markov transition operator from~\eqref{eq:markov_transition_operator}, 
	it is furthermore shown by \citet{Klus2019} that
	this empirical eigenproblem coincides exactly with the
	regularized eigenproblem given by the
	well-known kernel-based version of 
	\emph{extended dynamic mode decomposition} (EDMD)
	\citep{TRLBK14,WKR15,WRK15:Kernel}.
	Hence, the asymptotic viewpoint derived in our analysis proves
	that kernel EDMD essentially approximates
	$\ko: \inrkhs \rightarrow L^2(\invmeas)$
	in the infinite-sample limit with a suitable
	regularization scheme, thereby providing
	a statistical model for kernel EDMD.
	A theory of the spectral convergence of kernel EDMD
	could now be developed by investigating
	the spectral perturbation under the convergence
	$\norm{\empregA - \ko}_{\inrkhs \rightarrow L^2(\invmeas)} \to 0$
	for an admissible regularization scheme $\reg(n)$ and $n \to \infty$
	with suitable mixing assumptions of the underlying process
	along the lines of~\citet{Mollenhauer2020Autocovariance}.
	In particular, our approximation results from Section~\ref{sec:nonparametric_approximation}
	may be used to show that kernel EDMD overcomes the weak spectral
	convergence of standard EDMD which was
	proven by~\citet{KoMe18}.
	The details of this theory are not in the scope of this work
	and are subject to further research.

\section{Outlook}
\label{sec:outlook}

This work provides the theoretical framework for the nonparametric
approximation of the conditional expectation operator
$\ko$ over the RKHS embedded in its domain $\inrkhs \subset L^2(\nu)$ 
from an approximation viewpoint.
As a core result, we prove that convergence
takes place in the operator norm with respect to
the RKHS $\inrkhs$, 
therefore allowing for a stronger mode of convergence than
classically used numerical projection methods.

We establish the connection to
recent topics in statistical learning theory, in particular
least squares regression problems with vector-valued kernels and the
maximum mean discrepancy. These connections may allow
to extend our theory to practical applications such as nonparametric hypothesis tests
for Markov kernels.

In the case that $\ko$ is a Markov transition operator, our analysis
provides a statistical model for kernel-based EDMD.
However, in this case there remain open questions from
a theoretical perspective.
In particular, 
\begin{enumerate}[label=(\roman*)]
	\item convergence behaviour of the estimators need to
	be derived in terms of properties of the underlying Markov process
	such as a spectral gap, ergodicity rates and mixing;
	\item a spectral analysis of the estimators is needed in the context
	of classical perturbation theory in order
	to understand details of the spectral convergence.
\end{enumerate}

\section*{Acknowledgements}
	We greatfully acknowledge funding from Germany’s  
    Excellence  Strategy  (MATH+: The Berlin Mathematics 
    Research Center, EXC-2046/1, project AA1-2).
    The authors wish to thank Pietro Novelli and Vladimir Kostic
    for making us aware of an error in the original version
    of the manuscript.

\bibliographystyle{abbrvnat}

\bibliography{library}

\appendix

\section{}

It is well-known that for
$F \in L^2(E, \sigalg_E, \invmeas; \inrkhs)$,
the standard least squares
risk 
\begin{equation*}
	\risk(F) := \Ex \left[ 
	\norm{ \varphi(Y) - F(X) }_\inrkhs^2  \right],
\end{equation*}
can be rewritten in terms of the  \emph{regression function} $\Fstar$.
We report the proof here for completeness.

\begin{theorem}[Risk and regression function]
	\label{thm:risk_reformulation}
	Under Assumptions~\ref{ass:separability}--\ref{ass:second_moment},
	the risk $\risk$
	can equivalently be rewritten as
	\begin{equation}
		\risk(F) = \norm{ F - \Fstar }^2_{L^2(E, \sigalg_E, \invmeas; \inrkhs)}
		+ \risk(\Fstar)
	\end{equation}
	for all $F \in L^2(E, \sigalg_E, \invmeas; \inrkhs)$.
\end{theorem}

\begin{proof}
	We have
	\begin{align*} 
		\risk(F)
		&= \Ex \left[\norm{ \varphi(Y) - F(X) }_\inrkhs^2  \right]
		\\
		&=
		\Ex \left[
		\norm{ \varphi(Y) - \Fstar(X) + \Fstar(X) - F(X) }_\inrkhs^2  \right]
		\\
		&=
		\Ex \left[
		\norm{ \varphi(Y) - \Fstar(X) }_\inrkhs^2
		\right]
		\\
		&+ 2
		\Ex \left[
		\innerprod{\varphi(Y) - \Fstar(X)}{\Fstar(X) - F(X)}_\inrkhs
		\right]
		\\
		&+ 
		\Ex\left[ \norm{ F(X) - \Fstar(X) }^2_\inrkhs \right],
	\end{align*}
	where we see that the first summand equals to
	$\risk(\Fstar)$.
	The second summand which contains the mixed terms
	vanishes since we have
	\begin{align*}
		\Ex \left[
		\innerprod{\varphi(Y) - \Fstar(X)}{\Fstar(X) - F(X)}_\inrkhs
		\right] 
		\\
		= \int_E
		\innerprod{ 
			\underbrace{\int_E \varphi(y) \, \tk(x, \dd y) }_{ = \Fstar(x) }
			- \Fstar(x) }{ \Fstar(x) - F(x)}\strut_{\!\inrkhs} \, \dd \invmeas(x).
	\end{align*}
	The last summand can be rewritten as
	\begin{equation*}
		\Ex\left[ \norm{ F(X) - \Fstar(X) }^2_\inrkhs \right] 
		=
		\norm{ F - \Fstar }^2_{L^2(E, \sigalg_E, \invmeas; \inrkhs)}
	\end{equation*}
	by change of measure,
	proving the assertion.
\end{proof}

\end{document}